\documentclass[12pt,a4paper,reqno]{amsart}

\usepackage{float}
\usepackage{amsmath,mathrsfs}
\usepackage{amssymb}
\usepackage{color}
\usepackage{dsfont}
\usepackage[latin1]{inputenc}
\usepackage{graphicx}
\usepackage{hyperref}
\usepackage{enumitem}
\setitemize{leftmargin=*}

\newcommand\NoBlackBoxes{\global\overfullrule0pt}
\NoBlackBoxes

\makeatletter
\let\serieslogo@\relax
\let\@setcopyright\relax

\newtheorem{definition}{Definition}[section]
\newtheorem{theorem}[definition]{Theorem}

\newtheorem{proposition}[definition]{Proposition}
\newtheorem{rem}[definition]{Remark}

\renewcommand{\P}{{\mathbb{P}}}
\newcommand{\E}{{\mathbb{E}}}

\renewcommand{\epsilon}{\varepsilon}
\renewcommand{\phi}{\varphi}

\begin{document}

\setcounter{page}{1}

\title[]{The sparse Blume-Emery-Griffiths model of associative memories}

\author[Judith Heusel]{Judith Heusel}
\address[Judith Heusel]{Fachbereich Mathematik und Informatik,
University of M\"unster,
Einsteinstra\ss e 62,
48149 M\"unster,
Germany}

\email[Judith Heusel]{judith.heusel@uni-muenster.de}

\author[Matthias L\"owe]{Matthias L\"owe}
\address[Matthias L\"owe]{Fachbereich Mathematik und Informatik,
University of M\"unster,
Einsteinstra\ss e 62,
48149 M\"unster,
Germany}

\email[Matthias L\"owe]{maloewe@math.uni-muenster.de}


\date{\today}



\newcommand{\vep}{\varepsilon}
%

\begin{abstract}
We analyze the Blume-Emery-Griffiths (BEG) associative memory with sparse patterns and at zero temperature.
We give bounds on its storage capacity provided that we want the stored patterns to be
fixed points of the retrieval dynamics. We compare our results to that of other models of sparse neural networks and
show that the BEG model has a superior performance compared to them.
\end{abstract}

\maketitle

\section{Introduction}
The storage capacity of models of associative memories has been
intensively studied in the probability, physics, and, of course, neural networks
literature over the past more than thirty years. Since the seminal work of Hopfield \cite{Hopfield1982} his findings have been proven in a number of papers, see, e.g. \cite{AGS87} for an approach using the replica method, or \cite{MPRV}, \cite{Newman_hopfield}, \cite{Bov98}, \cite{L98}, \cite{talagrand} for methods using exponential estimates and large deviations techniques.
On the other hand, for sparse messages (where sparsity is defined by a small number of active neurons) related, yet different models of associative memories have been proposed by Willshaw \cite{Willshaw}, Amari \cite{Amari1989}, Okada \cite{Okada1996},  or Boll\'e and Verbeiren \cite{BV03}. This interest in associative memories was recently refreshed by Gripon, Berrou and coauthors in \cite{griponb}, \cite{griponc}, \cite{gripond}, \cite{gripone}, \cite{griponf}, also see \cite{HLV15} for a related model. Another reason for a deeper study of such models is that they are intrinsically related to search problems on big databases, see e.g. \cite{yu2015neural},\cite{griponlowevermet_search}.
All these models have in common that their storage capacity is conjectured to be much larger than that of the Hopfield model. The Willshaw model has also been discussed in a number of papers by Palm, Sommer, and coauthors (\cite{Palm1980}, \cite{Palm1996}, \cite{Palm2013} e.g.), with the difference that there the focus is rather on information capacity than on exact retrieval (and that many of the techniques are not rigorous). The conjecture of a superior storage capacity was proven to be true for Amari's model, the Willshaw model, the model from \cite{HLV15}, as well as one version of the models by Gripon and Berrou in \cite{GHLV16}, where also the performances of these models were compared. It was shown in \cite{heuseldiss} that the bounds on the storage capacities obtained in \cite{GHLV16} are indeed sharp. 

However, there is one model which was not considered in \cite{GHLV16}: the so-called Blume-Emery-Griffiths network. This network is a ternary network with input patterns from $\{-1,0,+1\}^N$.
It has been argued that, e.g. for character recognition,
sparse patterns, i.e. patterns with an intensity of less than
100\% for being $\pm 1$ are more realistic than equiprobable
binary ones. In \cite{BS92} a first model for such ternary
networks was introduced and studied. The authors come to the conclusion that
their model of a neural network can store about $\frac N{\gamma p
\log N}$ patterns, where $p$ is the activity of the patterns, i.e.
the probability that a fixed spin of a fixed pattern is different
from zero. This is in agreement with the findings in \cite{GHLV16}, that
the storage capacity of the networks increases, when the
activity is low.

On the other hand, it has been argued that an optimal
Hamiltonian, guaranteeing the best retrieval properties for neural
networks with multi-state neurons can be achieved by maximizing
the mutual information content of the networks (see e.g.
\cite{DK00}, \cite{BV03}). In this way, for two
state neurons the Hopfield model is retrieved while for the three
state problem described above one obtains the BEG model, that was originally introduced as a spin glass model to study phase separation of liquid He$^3$-He$^4$ mixtures. In \cite{BCS03} the
non-rigorous replica method is used to study the storage capacity
of such networks.

In \cite{LV_BEG} it has been rigorously proven that the BEG has a storage capacity for low activity patterns that outperforms that of the Hopfield model as well as that of the ternary model introduced in \cite{BS92}, if the activity does not depend on the system size.

The aim of the present note is to discuss the BEG model for very sparse data, i.e. data where the activity of the $N$ neurons is such that only $\log N$ of them are active in the mean. This is precisely the setup studied in \cite{GHLV16}. It will turn out that we need to adjust the model to this degree of sparsity, but also that after this adjustment the BEG model has a higher storage capacity than all the models discussed in \cite{GHLV16}.

The present note is organized in the following way. In Section 2 we will give a brief description of the situation we are in. We will also introduce two versions of the sparse BEG model. In Section 3 we will show that without adjustment the model proposed in \cite{LV_BEG} is not able to store a large amount of very sparse data. In Section 4 we will show that after the adjustment the model has a large storage capacity. We will also compare its capacity with other models of associative memories for sparse data analysed in \cite{GHLV16}.

\section{The Setup}
In this section we will describe the model we are interested in as
well as the notion of storage capacity we use.
An associative memory will always be a model operating on $N$ neurons,
$V=\{1, \ldots, N\}$. In this memory we want to store $M= M(N)$
patterns $(\xi_i^\mu)_{i=1,\ldots, N}^{\mu=1,\ldots, M}$,
which are taken as random elements from $\{-1,0,+1 \}^N$. The underlying probability
distribution of these patterns will be such that makes all the random variables
$\xi_i^\mu$ independent and identically distributed with
\begin{equation}
1-p=\P(\xi_i^\mu=0)\quad \mbox{ and }\quad
\P(\xi_i^\mu=1)=\P(\xi_i^\mu=-1)=\frac p 2.
\end{equation}
Here, different from the situation analyzed in \cite{LV_BEG} we will choose for the rest of the paper $p=p_N = \frac{\log N}N$, which is about the maximum the degree of sparsity one can allow for without obtaining patterns that are identically 0. As $p$ depends on $N$
we need to take a triangular array $(\xi^{\mu,N})$ of patterns rather than a sequence. However, as this detail does not play any role, we will suppress this dependency on $N$.

Correspondingly, each neuron carries a spin
$\sigma_i, i=1,\ldots, N$ that also takes its values in
$\{-1,0,+1 \}$. We will define the original BEG associative memory as in \cite{LV_BEG} by a dynamics on the spin space
$\{-1,0,+1 \}^N$.
This is given by the following
updating rule: update the spins $\sigma_i$ asynchronously at random using the dynamics
\begin{equation}\label{BEG2}
\tilde{T}_i(\sigma):=\text{sgn}\left(S_i(\sigma)\right)\Theta\left(\big\vert S_i(\sigma)\big\vert+\theta_i(\sigma)\right).
\end{equation}
Here $\Theta$ is the Heavyside-function (defined by $\Theta(x)=\mathds{1}_{[0,\infty)}(x)$),
$$ \theta_i(\sigma)=\sum_{j \neq i}K_{ij} \sigma^2_j ,\ \mbox{where}\ K_{ij}=\frac 1{(1-p)^2} \sum_{\mu=1}^M
\eta_i^\mu \eta_j^\mu,\ \mbox{and}\  \eta_i^\mu=
(\xi_i^\mu)^2-p,$$ and $S_i$ is the external field defined by
$$
S_i(\sigma):=\sum_{j\neq i}J_{ij}\sigma_j=\sum_{j\neq i}\sum_{\mu=1}^M\xi_i^\mu\xi_j^\mu\sigma_j \quad
\mbox{with }
J_{ij}:=\sum_{\mu=1}^M\xi^\mu_i\xi^\mu_j,\, i\neq j,\, i,j\in\lbrace1,\ldots,N\rbrace.
$$
Given this dynamics, we introduce the notion of storage capacity: To this end
the dynamics $T:=(T_i)_{i=1, \ldots, N}$ is taken as the
retrieval dynamics of the BEG memory, i.e. given
an input $\sigma$ the network will ``associate'' this input with
that pattern, that is found by (possibly many iterates of) $T$.
A minimum requirement for such an associative memory
is that the patterns themselves are stable under $T$, i.e. $T(\xi^\mu)=\xi^\mu$.
The storage capacity in this concept is defined
as the greatest number of patterns $M:=M(N)$ such that a randomly chosen
pattern $\xi^{\mu}$ is stable (of course,
this number depends on the randomly chosen patterns, such that in
the sequel we will speak about numbers $M(N)$ such that
with probability converging to one a pattern is
stable).

As we will see, the model as defined so far (and in \cite{LV_BEG}) does not have an impressive storage capacity, if we do not adapt it to the sparsity of the patterns. We will hence introduce a new dynamics of a similar form, including additionally a threshold term: the $i-$th component of $T$ assigns to neuron $i$ the value
\begin{equation}\label{BEG3}
T_i(\sigma):=\text{sgn}\left(S_i(\sigma)\right)\Theta\left(\big\vert S_i(\sigma)\big\vert+\theta_i(\sigma)-\gamma\log(N)\right).
\end{equation}
We will show that this last model outperforms all the models discussed in \cite{GHLV16}.

\section{Stability of the Stored Patterns in the Original BEG Model}

Let us first study the original BEG model as given by \eqref{BEG2} in our extremely sparse situation with $p_N=\log(N)/N$ and see that we indeed need to make some adjustment, to obtain a satisfactory result. To this end recall, that the storage capacity of all the models discussed in \cite{GHLV16}, thus Willshaw's model, Amari's model, as well as the model by Gripon and Berrou and the associative memory introduced in \cite{HLV15} share a memory capacity of $M=\alpha N^2/\log(N)^2$ messages (or patterns) with
different values for $\alpha$ for the different models. Hence assume that $M=\alpha N^2/\log(N)^2$. We will show that for the model defined by \eqref{BEG2} a randomly chosen message $\xi^\mu$ is not stable with a probability bounded away from zero. In fact, a randomly chosen inactive neuron $i$ of $\xi^\mu$ (i.e. $\xi_i^\mu=0$) is mapped to a non-zero value with positive probability.
The BEG model as in \eqref{BEG2} does thus not work well with this grade of sparsity without being modified.
\begin{theorem}
In the original BEG model \eqref{BEG2} with $M=\alpha N^2/\log(N)^2$ patterns, the stored patterns are instable with positive probability: that is,
\begin{align*}
\liminf_{N\rightarrow\infty}\mathbb{P}\left(\tilde{T}(\xi^\mu)\neq\xi^\mu\right)>0
\end{align*}
for any fixed but arbitrary $1\leq\mu\leq M$.
\end{theorem}
\begin{proof}
We consider message $\xi^1$. Without without loss of generality, $\xi^1$ has $k$ active neurons (i.e. $k$ of the $\xi_i^\mu$ are not $0$) and these are the first neurons $1,\ldots,k$. The corresponding event is denoted by $\mathcal{Z}_k$. It suffices to show that at least one of the inactive neurons is activated with non-vanishing probability. In fact, this is even true for an arbitrary neuron $i$, $i>k$.
To see this, we observe that an inactive neuron $i>k$ is mapped to a non-zero value if the Heaviside-function of $\vert S_i(\xi^1)\vert+\theta_i(\xi^1)$ is 1: this happens if
$\vert S_i(\xi^1)\vert+\theta_i(\xi^1)\geq0
$
which is equivalent to
$$\Big\vert\sum_{\mu=1}^M\sum_{j\neq i}\xi^1_j\xi^\mu_i\xi^\mu_j\Big\vert\geq-\frac{1}{(1-p)^2}\sum_{\mu=1}^M\sum_{j\neq i}\left(\xi^1_j\right)^2\eta^\mu_i\eta^\mu_j.
$$
We will show that with non-vanishing probability,
$\frac{1}{(1-p)^2}\sum_{\mu=1}^M\sum_{j\neq i}\left(\xi^1_j\right)^2\eta^\mu_i\eta^\mu_j>0$
which implies $\Theta(\vert S_i(\xi^1)\vert+\theta_i(\xi^1))=1$.
For the analysis
we use our assumption on the active neurons of $\xi^1$:
\begin{align*}
-&\frac{1}{(1-p)^2}\sum_{\mu=1}^M\sum_{j\neq i}\left(\xi^1_j\right)^2\eta^\mu_i\eta^\mu_j=\frac{kp}{1-p}-\frac{1}{(1-p)^2}\sum_{\mu=2}^M\sum_{j\leq k}\eta^\mu_i\eta^\mu_j
\end{align*}
which is negative if
$
kp(1-p)<\sum_{\mu=2}^M\sum_{j\leq k}\eta^\mu_i\eta^\mu_j,
$
the left hand side of which is smaller than 1, if $k\leq (1+\delta)\log(N)$ and $N$ is large enough. Since $$\mathbb{P}(A_\delta):=\mathbb{P}\left((1-\delta)\log(N)<\sum_{j=1}^N\vert\xi^1_j\vert<(1+\delta)
\log(N)\right)
\longrightarrow1
$$ as $N$ tends to infinity, it suffices to show that
\begin{align}\label{BEG erste Version dritte Gleichung}
\liminf_{N\rightarrow\infty}\min_{\substack{k\in\mathbb{N}:k/\log(N)\\\in(1-\delta,1+\delta)}}\mathbb{P}\left(\sum_{\mu=2}^M\sum_{j\leq k}\eta^\mu_N\eta^\mu_j\geq1\right) >0.
\end{align}
To analyse this expression let $U(N)$ be the random variable
$U(N):=\sum_{\mu=2}^M\vert \xi^\mu_N\vert.
$
$U(N)$ is Binomially distributed with parameters $M-1=\alpha\frac{N^2}{\log(N)^2}-1$ and $p$. As $M$ is large we may replace $M-1$ by $M$. Using Chebyshev's inequality, we obtain
$$\mathbb{P}\left(\frac{(1-\delta)N\alpha}{\log(N)}<U(N)<\frac{(1+\delta)N\alpha}{\log(N)}\right)
\longrightarrow1
$$
for each $\delta>0$ as $N$ tends to infinity. Let us denote
$\lbrace\frac{U(N)\log(N)}{N\alpha}\in(1-\delta,1+\delta)\rbrace=:B_\delta(N).$
It then suffices to consider the set $B_\delta(N)$. Indeed, we observe
\begin{align*}
\mathbb{P}\left(\sum_{\mu=2}^M\sum_{j\leq k}\eta^\mu_N\eta^\mu_j\geq1\right)\geq
\mathbb{P}( B_\delta(N))\min_{\substack{m\in\mathbb{N}:m\log(N)/\\(N \alpha )\in(1-\delta,1+\delta)}}
\mathbb{P}\left[\sum_{\mu=2}^M
\sum_{j\leq k}\eta^\mu_N\eta^\mu_j\geq1\Big\vert U(N)=m\right].
\end{align*}
For neuron $N$, we define the following two random variables:
$$V(k,N):=\sum_{\mu:\vert\xi^\mu_N\vert=1}\quad\sum_{j=1}^k\vert\xi_j^\mu\vert
\qquad \mbox{and }
W(k,N):=\sum_{\mu:\vert\xi^\mu_N\vert=0}\quad\sum_{j=1}^k\vert\xi_j^\mu\vert.
$$
We fix $k$, omit the reference to the dependence on $N$ and $k$ in the next computation, and write $V$, $W$ and $U$, instead. The sum in (\ref{BEG erste Version dritte Gleichung}) can be rewritten as
\begin{align}\label{Umformung BEG erste Version}
&\sum_{\mu=2}^M\sum_{j=1}^k\eta^\mu_N\eta^\mu_j
=\sum_{\mu:\vert\xi^\mu_N\vert=1}\text{ }\sum_{j=1}^k(1-p)\eta^\mu_j+\sum_{\mu:\vert\xi_N^\mu\vert=0}\text{ }\sum_{j=1}^k(-p)\eta^\mu_j\nonumber\\
=&V\left(1-p\right)^2+(Uk-V)(-p)(1-p)+W(1-p)(-p)+\left((M-U)k-W\right)p^2\nonumber\\
=&V-Vp-Ukp-Wp+Mp^2k
=V-Vp-Ukp-Wp+\alpha k.
\end{align}
Assume that $U(N)=m$. Then $V(k,N)$ is Binomially distributed with parameters $km$ and $p$. We use the Berry-Esseen-bound (see \cite{durrett_book} Theorem 3.4.9) to obtain
$$\Bigg\vert\mathbb{P}\left(\frac{V(k,N)-kmp}{\sqrt{p(1-p)km}}\leq b \Big\vert U(N)=m\right)-\Phi(b)\Bigg\vert\leq\frac{Cp(1-p)\left((1-p)^2+p^2\right)}{p(1-p)\sqrt{p(1-p)km}}
$$
where $\Phi(\cdot)$ is the distribution function of the standard normal distribution.
For $k,m$ taken from the sets we consider, i.e. fulfilling
$$\vert k-\log(N)\vert \leq \delta \log(N),\quad \vert m-\alpha N/\log(N)\vert \leq \delta \alpha N/\log(N),
$$
the right hand side is at most equal to
$$\frac{Cp(1-p)\left((1-p)^2+p^2\right)}{p(1-p)\sqrt{p(1-p)km}}\leq \frac{C\left((1-p)^2+p^2\right)}{\sqrt{(1-p)\alpha (1-\delta)^2\log(N)}}
$$
and therefore vanishing as $N$ tends to infinity. In particular, this implies
\begin{align*}
&\mathbb{P}\left(V(k,N)\in\left(kmp+0.1\sqrt{kmp(1-p)},kmp+3\sqrt{kmp(1-p)}\right)\Big \vert U(N)=m\right)\\
\geq&\Phi(3)-\Phi(0.1)-2C\frac{1-2p+2p^2}{\sqrt{(1-p) \alpha (1-\delta)^2 \log(N)}}
\end{align*}
for the above choice of $k$ and $m$.

The same line of arguments holds for $W(k,N)$: Conditionally on $\lbrace U(N)=m\rbrace$, $W(k,N)$ is Binomially distributed with parameters $(M-1-m)k$ and $p$. Without loss of generality, we again replace $(M-1-m)k$ by $(M-m)k$. The Berry Esseen theorem in this case gives
$$\Bigg\vert\mathbb{P}\left(\frac{W(k,N)-(M-m)kp}{\sqrt{p(1-p)(M-m)k}}\leq b \Big\vert U(N)=m\right)-\Phi(b)\Bigg\vert\leq\frac{Cp(1-p)\left((1-p)^2+p^2\right)}{\sqrt{p^3(1-p)^3(M-m)k}}
$$
and this also tends to 0 for our choice of $k$ and $m$.
 We conclude
\begin{align*}
\mathbb{P}&\left(  W(k,N)\in\Big((M-m)kp-3\sqrt{(M-m)kp(1-p)},(M-m)kp-\right.\\
&\left.\quad\quad0.1\sqrt{(M-m)kp(1-p)}\quad\Big\vert U(N)=m\right)\\\geq&\Phi(-0.1)-\Phi(-3)-2C\frac{1-2p+2p^2}{\sqrt{\log(N)(1-\delta)  \alpha (1-p)[N/\log(N)-1-\delta]}}.
\end{align*}
We will see that for an arbitrary but fixed choice of $k$ and $m$, such that $k$ and $m$ are chosen as above, and with
$$F(k,m,N):=\Big\lbrace V(k,N)\in\left(kmp+0.1\sqrt{kmp(1-p)},kmp+3\sqrt{kmp(1-p)}\right)\Big\rbrace
$$
and
\begin{align*}
G(k,m,N):=&\Big\lbrace  W(k,N)\in\Big((M-m)kp-3\sqrt{(M-m)kp(1-p)},
\\&\quad\quad
(M-m)kp-0.1\sqrt{(M-m)kp(1-p)}\Big)\Big\rbrace,
\end{align*}
we have
\begin{equation}\label{prec_claim}
\lbrace U(N)=m\rbrace \cap \Big\lbrace\sum_{\mu=2}^M
\sum_{j\leq k}\eta^\mu_N\eta^\mu_j\geq1\Big\rbrace \supseteq \lbrace U(N)=m\rbrace \cap  F(k,m,N)\cap G(k,m,N).
\end{equation}
This yields
\begin{align*}
&\liminf_{N\rightarrow\infty}\mathbb{P}\left(\exists i\leq N:\xi^1_i=0, T_i(\xi^1)\neq 0\right)
\\ \geq&
 \liminf_{N\rightarrow\infty}\mathbb{P}(A_\delta)\mathbb{P}(B_\delta(N))\min_{\substack{k,m\in\mathbb{N}:k/\log(N),\\m\log(N)/(\alpha N)\in(1-\delta,1+\delta)}}\mathbb{P}\left[\sum_{\mu=2}^M\sum_{j\leq k}\eta^\mu_N\eta^\mu_j\geq1\Big\vert  U(N)=m\right]\\
\geq&\liminf_{N\rightarrow\infty}\mathbb{P}(A_\delta)\mathbb{P}(B_\delta(N))\min_{\substack{k,m\in\mathbb{N}:k/\log(N),\\m\log(N)/(\alpha N)\in(1-\delta,1+\delta)}}\mathbb{P}\left[F(k,m,N) \cap G(k,m,N)\Big\vert   U(N)=m\right]
\end{align*}
which is larger than $0$.
To show \eqref{prec_claim}, we write
$$U(N)=m=\rho_2\alpha N/\log(N),\quad k=\rho_1\log(N),$$ with $\rho_1,\rho_2\in(1-\delta,1+\delta)$, as well as
$$V(k,N)=\rho_1\rho_2\alpha \log(N)+\rho_3\sqrt{\rho_1\rho_2\alpha\log(N)(1-p)}$$ and $$W(k,N)=\alpha\rho_1\left[ N-\rho_2 \log(N)\right]-\rho_4 \sqrt{\alpha\rho_1\left[ N-\rho_2 \log(N)\right](1-p)}$$
again $\rho_3,\rho_4\in(0.1,3)$. We transform the right hand side of (\ref{Umformung BEG erste Version}) in the following way:
\begin{align*}
&V(k,N)-V(k,N)p-U(N)kp-W(k,N)p+\alpha k\\
=&\rho_3\sqrt{\rho_1\rho_2\alpha\log(N)(1-p)}-
\frac{\log(N)^2}{N}\rho_1\rho_2\alpha -\rho_3\frac{\log(N)}{N}\sqrt{\rho_1\rho_2\alpha \log(N)(1-p)}\\
&-\alpha\rho_1\log(N)+\frac{\log(N)^2}{N}\rho_1\rho_2\alpha +\frac{\log(N)}{N}\rho_4 \sqrt{\alpha\rho_1\left[ N-\rho_2 \log(N)\right](1-p)}
+\alpha\rho_1\log(N)\\
=&\rho_3\sqrt{\rho_1\rho_2\alpha\log(N)(1-p)}+O\left(\frac{\log(N)}{\sqrt{N}}\right)\geq 1
\end{align*}
if $N$ is large enough.

Consequently, with positive probability not converging to zero, a randomly chosen inactive neuron of message $\xi^1$ is turned into a 1 or a -1 and $\xi^1$ is not a fixed point of the dynamics.
\end{proof}

\begin{rem}
It is exactly the sparsity of the patterns that poses the problems with the dynamics given by \eqref{BEG2}. For a pattern used as input and an inactive neuron $i$ of the pattern, the part of $\theta_i$ coming from the stored pattern itself: on the one hand, there are only a few active neurons, on the other hand, $p$ is very small and the resulting threshold can easily be exceeded by the noise term of $\theta_i$.
\end{rem}

\section{The Storage Capacity of the BEG Model as defined in \eqref{BEG3}}

To cope with the sparsity of the patterns, we add a threshold and change the dynamics of the BEG model into \eqref{BEG3} with a value of $\gamma >0$ to be chosen.
A main part of the current sections consists of checking whether with this new dynamics we can store an amount
$M=\alpha N^2/\log(N)^2$ patterns in the model in such a way that an arbitrary one is stable with high probability if $\alpha$ is appropriately chosen and $N$ tends to infinity:
\begin{theorem}\label{BEG Kapazitaet}
In the BEG network with $M=\alpha N^2/\log(N)^2$ the dynamics defined in \eqref{BEG3} satisfies
$$\lim_{N\rightarrow\infty}\mathbb{P}\left(\exists i\leq N:
T_i(\xi^\mu)\neq\xi^\mu_i\right)=0
$$
for any arbitrary but fixed $\mu$, if the following conditions are fulfilled:
$$0<\gamma<2 \qquad \mbox{and }\quad
\alpha<\frac{\gamma}{x^*_\gamma-1}
$$
with $x^*_\gamma$ being the unique root of the function
$$ g_\gamma(x):=x\left(1+\frac{2}{\gamma}-\log(x)\right)-1-\frac{2}{\gamma}
$$
in $(1,\infty)$. This bound is sharp: If
$\alpha>\frac{\gamma}{x^*_\gamma-1},
$
an arbitrary stored pattern is instable with high probability:
$$\lim_{N\rightarrow\infty}\mathbb{P}\left(\exists i\leq N:
T_i(\xi^\mu)\neq\xi^\mu_i\right)=1.
$$
\end{theorem}
\begin{proof}
Without loss of generality, we consider message $\xi^1$.
First, observe that there are two principal types of errors which can occur, namely:
\begin{itemize}
\item
a 0 is turned into a 1 or to a -1
\item
a non-zero spin is turned to a 0 or multiplied by (-1).
\end{itemize}
For $\delta >0$ let us denote the event $\{\sum_{j=1}^N \vert \xi_j^1\vert/ \log(N) \in (1 - \delta; 1 + \delta)\}$ by $A_\delta$. 
Keeping the notation of the previous section,
intersecting the event that spin $i$ of $\xi^1$ is not recovered correctly by the dynamics with $A_\delta$ and
denoting the active spins of $\xi^1$ by $1, \ldots , k$ we obtain
\begin{align*}
&\mathbb{P}\left(\exists i\leq N: T_i(\xi^1)\neq\xi^1_i\right)
\leq\mathbb{P}(A_\delta)\max_{\substack{k\in\mathbb{N}:k/\log(N)\\
\in(1-\delta,1+\delta)}}\mathbb{P}\left(\exists i\leq N: T_i(\xi^1)\neq\xi^1_i\vert \mathcal{Z}_k\right)+\mathbb{P}(A_\delta^c).
\end{align*}
The probability $\mathbb{P}(A_\delta^c)$ tends to 0 and it suffices to examine the conditional probabilities, given $\mathcal{Z}_k$, for $k$ belonging to the set mentioned above.

We begin with the analysis of the first kind of error: to this purpose, we fix $k$ and take some $i\geq k+1$, e.\,g., $i=k+1$. An error in this place occurs if $\vert S_{k+1}(\xi^1)\vert+\theta_{k+1}(\xi^1)-\gamma \log(N)\geq0$, i.e., if
\begin{align}\label{BEG first case inequality}
\Bigg\vert\sum_{j\neq k+1}\sum_{\mu=1}^M\xi^1_{j}\xi_{k+1}^\mu\xi_j^\mu\Bigg\vert >-\frac{1}{(1-p)^2}\sum_{j\neq k+1}\left(\xi^1_j\right)^2\sum_{\mu=1}^M\eta_{k+1}
^\mu\eta_j^\mu+\gamma\log(N).
\end{align}
Since we consider $i=k+1$ and $\mathcal{Z}_k$, i.e., $\xi^1_{k+1}=0$ and additionally $\xi^1_j=0$ for $j>k+1$, the left hand side of (\ref{BEG first case inequality}) is equal to
$$\Bigg\vert\sum_{j\neq k+1}\sum_{\mu=1}^M\xi^1_{j}\xi_{k+1}^\mu\xi_j^\mu\Bigg\vert=
\Bigg\vert\sum_{j\neq k+1}\xi_j^1\xi^1_{j}\xi_{k+1}^1+\sum_{j=1}^k\sum_{\mu=2}^M\xi^1_{j}\xi_{k+1}^\mu\xi_j^\mu
\Bigg\vert=
\Bigg\vert\sum_{j=1}^k\sum_{\mu=2}^M\xi^1_{j}\xi_{k+1}^\mu\xi_j^\mu
\Bigg\vert.
$$
On the other hand, for the random part of the right-hand side of (\ref{BEG first case inequality}) observe that
\begin{align*}
&-\sum_{j\neq k+1}\left(\xi^1_j\right)^2\sum_{\mu=1}^M\eta_{k+1}^\mu
\eta_j^\mu
=
-\sum_{j=1}^{k}\left[\eta_{k+1}
^1\eta_j^1+\sum_{\mu=2}^M\eta_{k+1}
^\mu\eta_j^\mu\right]\\
=&-\sum_{j=1}^{k}\left[\left(-p\right)\left(1-p\right)+\sum_{\mu=2}^M\eta_{k+1}
^\mu\eta_j^\mu\right]
=kp(1-p)-\sum_{j=1}^{k}\sum_{\mu=2}^M\eta_{k+1}
^\mu\eta_j^\mu.
\end{align*}
Hence(\ref{BEG first case inequality}) becomes
$$
\Bigg\vert\sum_{j=1}^k\sum_{\mu=2}^M\xi^1
_{j}\xi_{k+1}^\mu\xi_j^\mu\Bigg\vert>\frac{kp}{(1-p)}-\frac{1}{(1-p)^2}\sum_{j=1}^{k}\sum_{\mu=2}^M\eta_{k+1}
^\mu\eta_j^\mu+\gamma\log(N).
$$
This is either fulfilled if
$$\sum_{j=1}^k\sum_{\mu=2}^M\left[\xi^1_{j}\xi_{k+1}^\mu\xi_j^\mu+\frac{1}{(1-p)^2}\eta_{k+1}
^\mu\eta_j^\mu\right]>\frac{kp}{(1-p)}+\gamma\log(N)
$$
or if
$$\sum_{j=1}^k\sum_{\mu=2}^M\left[-\xi^1_{j}\xi_{k+1}^\mu\xi_j^\mu+\frac{1}{(1-p)^2}\eta_{k+1}
^\mu\eta_j^\mu\right]>\frac{kp}{(1-p)}+\gamma\log(N).
$$
Now for i.i.d. random variables $Z_1,Z_2,Z_3$ with $\mathbb{P}(Z_1=\pm 1)=\frac{1}{2}$, the events $\lbrace Z_1Z_2Z_3=x_3\rbrace$ and $\lbrace Z_1=x_1,Z_2=x_2\rbrace$ are independent for each choice of $x_1,x_2,x_3\in\lbrace-1,1\rbrace$. The same is true for $\lbrace Z_1Z_2Z_3=x_3\rbrace$ and $\lbrace  Z_1=x_1\rbrace$ for arbitrary $x_1,x_3\in\lbrace-1,1\rbrace$ as well as for $\lbrace Z_1Z_2=x_2\rbrace$ and $\lbrace  Z_1=x_1\rbrace$ for arbitrary $x_1,x_2\in\lbrace-1,1\rbrace$. In particular, the conditional distribution of $Z_1 Z_2 Z_3$ with respect to an arbitrary realisation of $Z_1Z_2$ is the same as the distribution of $Z_3$. Similarly, the conditional distribution of $Z_1Z_2Z_3$ with respect to $\lbrace Z_1=x_1\rbrace$ is the distribution of $Z_2Z_3$ and the one of $Z_1Z_2$ with respect to $\lbrace Z_1=x_1\rbrace$ is the same as the one of $Z_2$, for arbitrary $x_1\in\lbrace-1,1\rbrace$.
So, for fixed $i$, the conditional distribution of $\xi^1_{j'}\xi^\mu_i\xi^\mu_{j'}$, for $j'\neq i$, given $\xi^1_j$, $1\leq j\leq N$, is
$$\mathbb{P}\left(\xi^1_{j'}\xi^\mu_i\xi^\mu_{j'}=\pm1\big\vert\thickspace \xi^1_j, j\leq N\right)=\frac{p^2}{2}\vert \xi^1_{j'}\vert, \quad \mathbb{P}\left(\xi^1_{j'}\xi^\mu_i\xi^\mu_ {j'}=0\big\vert\thickspace \xi^1_j, j\leq N\right)=1-\vert \xi^1_{j'}\vert p^2
$$
which is the distribution of $\xi^\mu_i\xi^\mu_{j'}$, if $\vert \xi^1_{j'}\vert =1$. Given an arbitrary realisation of $\xi^1_j,j\leq N$, and for fixed $i$,
$(\xi^1_j\xi^\mu_i\xi^\mu_j,  j\neq i, \mu \geq 2)\text{ and }(\vert \xi^1_j\vert \xi^\mu_i\xi^\mu_j,  j\neq i, \mu \geq 2)$ are identically distributed and we can conclude that conditionally on $\mathcal{Z}_k$, the sum $\sum_{ j\neq i}^N\sum_{\mu=2}^M\xi^1_{j}\xi_{i}^\mu\xi_j^\mu$ has the same distribution as
$\sum_{ j\neq i}^k\sum_{\mu=2}^M\xi_{i}^\mu\xi_j^\mu$.
In addition, conditionally on $\mathcal{Z}_k$, we have for an arbitrary $i$
\begin{eqnarray*}
\sum_{j=1, j\neq i}^N\sum_{\mu=2}^M\xi^1_{j}\xi_{i}^\mu\xi_j^\mu+\frac{1}{(1-p)^2}\left(\xi^1_j\right)^2\eta_i^\mu\eta_j^\mu &\sim&\sum_{j\leq k,j\neq i}\sum_{\mu=2}^M\xi_{i}^\mu\xi_j^\mu+\frac{1}{(1-p)^2}\eta_i^\mu\eta_j^\mu\\
 &\sim& \sum_{j\leq k,j\neq i}\sum_{\mu=2}^M-\xi_{i}^\mu\xi_j^\mu+\frac{1}{(1-p)^2}\eta_i^\mu\eta_j^\mu
\end{eqnarray*}
(where $\sim$ denotes equality in distribution)
because the sums
$\sum_{j=1,j\neq i}^N\sum_{\mu=2}^M\left(\xi^1_j\right)^2\eta^\mu_i\eta^\mu_j$
and $\sum_{j=1,j\neq i}^k\sum_{\mu=2}^M\eta^\mu_i\eta^\mu_j$ are measurable with respect to $\vert\xi^\mu_j\vert$, $1\leq j\leq N$, $1\leq\mu\leq M$ and the sums
$\sum_{j=1,j\neq i}^N\sum_{\mu=2}^M\xi^1_j\xi_{i}^\mu\xi_j^\mu$ and $\sum_{j=1,j\neq i}^k\sum_{\mu=2}^M\xi_{i}^\mu\xi_j^\mu
$ are symmetrically distributed with respect to these random variables.
Using these considerations we obtain for the inactive neurons of $\xi^1$:
\begin{align*}
&\mathbb{P}\left[\exists i\geq k+1:T_i(\xi^1)\neq0\vert\mathcal{Z}_k\right]\\
\leq& N\mathbb{P}\left[T_{k+1}(\xi^1)\neq0\vert\mathcal{Z}_k\right]
\leq 2N\mathbb{P}\left[\sum_{j\leq k}\sum_{\mu=2}^M\xi_{k+1}^\mu\xi_j^\mu+\frac{1}{(1-p)^2}\eta_{k+1}^\mu\eta_j^\mu>\frac{kp}{1-p}+\gamma\log(N)\right].
\end{align*}
Next we use the exponential Chebyshev inequality: For $t>0$,
\begin{align*}
\mathbb{P}&\left[\sum_{j\leq k}\sum_{\mu=2}^M\xi_{k+1}^\mu\xi_j^\mu+\frac{1}{(1-p)^2}\eta_{k+1}^\mu\eta_j^\mu>\frac{kp}{1-p}+\gamma\log(N)\right]\\
\leq\exp&\left[-t\left(\frac{kp}{1-p}+\gamma\log(N)\right)\right]
\mathbb{E}\left[
\exp\left(t\sum_{j\leq k}\sum_{\mu=2}^M\xi_{k+1}^\mu\xi_j^\mu+\frac{1}{(1-p)^2}\eta_{k+1}^\mu\eta_j^\mu\right)\right].
\end{align*}
As the $M$ messages are independent the expectation on the right hand side becomes
\begin{align*}
&\mathbb{E}\left[\exp\left(t\sum_{j\leq k}\xi_{k+1}^M\xi_j^M+\frac{1}{(1-p)^2}\eta_{k+1}^M\eta_j^M\right)\right]^{M-1}\\
=&\left[(1-p)\cdot\left[(1-p)e^{t\frac{p^2}{(1-p)^2}}+pe^{-t\frac{p}{1-p}}\right]^{k}+p\cdot\left[(1-p)e^{-t\frac{p}{1-p}}+\frac{1}{2}pe^{2t}+\frac{1}{2}p\right]^{k}\right]^{M-1}.
\end{align*}
Recall, that we are considering $k\in\mathbb{N}$ with $\frac{k}{\log(N)}\in(1-\delta,1+\delta)$. Expanding the exponential and employing the Binomial formula, we obtain for $t$ not depending on $N$:
\begin{align}\label{exp EW BEG komplett}
&\left[(1-p)\cdot\left((1-p)e^{t\frac{p^2}{(1-p)^2}}+pe^{-t\frac{p}{1-p}}\right)^{k}+p\cdot\left((1-p)e^{-t\frac{p}{1-p}}+\frac{1}{2}pe^{2t}
+\frac{1}{2}p\right)^{k}\right]^{M-1}\nonumber\\
=&\left[(1-p)\left[(1-p)\left(1+\frac{tp^2}{(1-p)^2}+\mathcal{O}(p^4)\right)+p\left(1-\frac{tp}{1-p}+\mathcal{O}(p^2)\right)\right]^k\right.\nonumber\\
&\quad\quad\quad\left.+p\left[(1-p)\left(1-\frac{tp}{1-p}+\mathcal{O}(p^2)\right)+\frac{1}{2}pe^{2t}+\frac{1}{2}p\right]^k\right]^{M-1}\nonumber\\
=&\left[(1-p)\left[1+\mathcal{O}(p^3)\right]^k+p\left[1-\frac{tp}{1-p}-\frac{1}{2}p+\frac{1}{2}pe^{2t}+\mathcal{O}(p^2)\right]^k\right]^{M-1}\nonumber\\
=&\left[1+kp^2\left(\frac{1}{2}e^{2t}-\frac{1}{2}-\frac{t}{1-p}\right)+\mathcal{O}(p^3k^2)\right]^{M-1}\nonumber\\
\leq&\exp\left[k\alpha\left(\frac{1}{2} e^{2t}-\frac{1}{2}-\frac{t}{1-p}\right)+\mathcal{O}\left(Mp^3k^2\right)\right]
\le \exp\left[k\alpha\left(\frac{1}{2} e^{2t}-\frac{1}{2}-\frac{t}{1-p}\right)+\mathcal{O}\left(pk^2\right)\right].
\end{align}
In the last steps we used the inequality $1+x\leq e^x$ for each $x\in\mathbb{R}$ and inserted our choice for the number of stored patterns $M=\alpha N^2/\log(N)^2$. In combination with the previous steps of the proof, the following conditional probability is at most
\begin{align*}
&\mathbb{P}\left(\exists i\geq k+1:T_i(\xi^1)\neq0\vert\mathcal{Z}_k\right)\\
\leq&\exp\left(\log(N)+\log(2)-\frac{tkp}{1-p}-t\gamma \log(N)+k\alpha\left(\frac{1}{2} e^{2t}-\frac{1}{2}-\frac{t}{1-p}\right)\right)\left(1+\mathcal{O}\left(pk^2\right)\right).
\end{align*}
We need to show that the second line vanishes for $\frac{k}{\log(N)}\in(1-\delta,1+\delta)$.
Anticipating that $t$ will not depend on $N$ and writing $k=\rho \log(N)$ for $\rho\in(1-\delta,1+\delta)$, the exponent is equal to
\begin{align}\label{expBEG}
\log(N)-t\gamma \log(N)+\rho\log(N)\alpha\left(\frac{1}{2} e^{2t}-\frac{1}{2}-t\right)+o(\log(N)).
\end{align}
 Thus $\mathbb{P}\left(\exists i\leq N:\xi^1_i=0,T_i(\xi^1)\neq0\right)\to 0$, if, for each $\rho\in(1-\delta,1+\delta)$ , there is some $t_{\alpha,\gamma,\rho}>0$ depending on $\rho$, $\gamma$ and $\alpha$ such that
$$
h_{\alpha,\gamma,\rho}(t_{\alpha,\gamma,\rho}):=-t_{\alpha,\gamma,\rho}\gamma+\rho\alpha\left(\frac{1}{2} e^{2t_{\alpha,\gamma,\rho}}-\frac{1}{2}-t_{\alpha,\gamma,\rho}\right)<-1.
$$
 Using the global minimum of $h_{\alpha,\gamma,\rho}$, i.e. 
$t^*_{\alpha,\gamma,\rho}:=\frac{1}{2}\log\left(1+\frac{\gamma}{\rho\alpha}\right), t^*_{\alpha,\gamma,\rho}>0
$
 the dominant term in (\ref{expBEG}) is equal to
\begin{align*}
&\log(N)-t_{\alpha,\gamma,\rho}^*\gamma \log(N)+\rho\log(N)\alpha\left(\frac{1}{2} e^{2t^*_{\alpha,\gamma,\rho}}-\frac{1}{2}-t^*_{\alpha,\gamma,\rho}\right)\\
=&\log(N)\left[1-\gamma\frac{1}{2}\log\left(1+\frac{\gamma}{\rho\alpha}\right)+\frac{1}{2}\rho\alpha\left(1+\frac{\gamma}{\rho\alpha}\right)-\frac{1}{2}\rho\alpha-\frac{1}{2}\rho\alpha\log\left(1+\frac{\gamma}{\rho\alpha}\right)\right]\\
=&\log(N)f_{\alpha,\rho}(x_{\alpha,\gamma,\rho})
\end{align*}
with
$f_{\alpha,\rho}(x):=1+\frac{1}{2}\rho\alpha\left(-x\log(x)+x-1\right)
\mbox{ and }
x_{\alpha,\gamma,\rho}:=1+\frac{\gamma}{\rho\alpha}.$
We thus need to have that
$f_{\alpha,\rho}(x_{\alpha,\gamma,\rho})
<0$
for each $\rho\in(1-\delta,1+\delta)$.
For fixed $\gamma$ and $\alpha$, $f_{\alpha,\rho}(x_{\alpha,\gamma,\rho})$ is continuous in $\rho\in\mathbb{R}_+$.
If for fixed $\gamma,\alpha$, the inequality $f_{\alpha,1}(x_{\alpha,\gamma,1})<0$ holds true, then also
$f_{\alpha,\rho}(x_{\alpha,\gamma,\rho})<0$ for all $\rho\in(1-\delta,1+\delta)$
if $\delta>0$ is small enough. Therefore, $f_{\alpha,1}(x_{\alpha,\gamma,1})<0$, implies $\mathbb{P}\left(\exists i\leq N:\xi^1_i=0,T_i(\xi^1)\neq0\right)\to 0$ as desired.
It thus remains to determine the set $\lbrace \alpha>0: f_{\alpha,1}(x_{\alpha,\gamma,1})<0\rbrace$ for fixed $\gamma$. Multiplying by $2/\alpha$ and replacing $1/\alpha$ by $(x_{\alpha,\gamma,1}-1)/\gamma$ allows to reformulate the condition $f_{\alpha,1}(x_{\alpha,\gamma,1})<0$ in terms of the function $ g_\gamma(x_{\alpha,\gamma,1})$ defined in Theorem~\ref{BEG Kapazitaet}. $f_{\alpha,1}(x_{\alpha,\gamma,1})<0$ holds if
\label{page 118}
$$ g_\gamma(x_{\alpha,\gamma,1}):=x_{\alpha,\gamma,1}\left(1+\frac{2}{\gamma}-\log(x_{\alpha,\gamma,1})\right)-1-\frac{2}{\gamma}<0.
$$
The function $g_{\gamma}:(0,\infty)\rightarrow\mathbb{R}$ has two roots, one is equal to 1, the other one is bigger than 1 (and bigger than the extremal point $\hat{x}_\gamma=e^{2/\gamma}$). The derivative $g_\gamma'(x)=2/\gamma-\log(x)$ is positive on $(0,\hat{x}_\gamma)$ and negative on $(\hat{x}_\gamma,\infty)$. Let $x^*_\gamma$ be the unique root of $g_\gamma$ bigger than 1. The function $g_\gamma$ is negative on the intervals $(0,1)$ and $(x^*_\gamma,\infty)$ and positive on $(1,x^*_\gamma)$. We choose $\alpha>0$ depending on $\gamma$ such that
\begin{align}
\alpha<\frac{\gamma}{x^*_\gamma-1}.
\end{align}
Then each pair of $\gamma>0$ and $\alpha>0$ chosen subject to these conditions provides
$$\mathbb{P}\left[\exists i\leq N
:\xi^1_i=0,T_i(\xi^1)\neq0\right]\longrightarrow0 \qquad \mbox{as $N$ tends to infinity.}
$$
Due to conditions following from subsequent calculations, $\gamma$ must additionally be chosen such that
$\gamma\in(0,2).$

We now analyse the second type of error. We choose an arbitrary neuron that is activated in $\xi^1$; assuming again that $\mathcal{Z}_k$ holds, we take $i=1$. The neuron is not mapped to its original value if either
\begin{align}\label{BEG second type error first inequality}
\vert S_1(\xi^1)\vert+\theta_1(\xi^1)-\gamma\log(N)<0
\end{align}
or if the two conditions
\begin{align}\label{BEG second type error second inequality}
\vert S_1(\xi^1)\vert+\theta_1(\xi^1)-\gamma\log(N)>0
\end{align}
and
\begin{align}\label{BEG second type error third inequality}
\text{sgn}(S_1(\xi^1))\neq\xi^1_1
\end{align}
hold.
Without loss of generality, we set $\xi^1_1=1$.
Inequality (\ref{BEG second type error first inequality}) is either satisfied if
\begin{align}\label{BEG inequality S Teil 1}\vert S_1(\xi^1)\vert=S_1(\xi^1),\text{  }S_1(\xi^1)+\theta_1(\xi^1)-\gamma\log(N)<0
\end{align}
or if
\begin{align}\label{BEG inequality S Teil 2}
\vert S_1(\xi^1)\vert=-S_1(\xi^1),\text{  }-S_1(\xi^1)+\theta_1(\xi^1)-\gamma\log(N)<0.
\end{align}
 Condition (\ref{BEG second type error third inequality}) is necessary for (\ref{BEG inequality S Teil 2}). So the probability of having an error, given $\mathcal{Z}_k \cap\{\xi^1_1=1\}:= \bar{\mathcal{Z}}_k$, is bounded by the sum of the two probabilities
\begin{align}\label{BEG second type error first probability}
\mathbb{P}\left(S_1(\xi^1)+\theta_1(\xi^1)-\gamma\log(N)<0,\vert S_1(\xi^1)\vert=S_1(\xi^1)\Big\vert\bar{\mathcal{Z}}_k
\right)
\end{align}
and
\begin{align}\label{BEG second typ error second}
\mathbb{P}\left(\text{sgn}(S_1(\xi^1))\neq1\vert\bar{\mathcal{Z}}_k
\right).
\end{align}
The probability in (\ref{BEG second typ error second}) can be written as
\begin{align*}
&\mathbb{P}\left[\text{sgn}(S_1(\xi^1))\neq 1\vert\bar{\mathcal{Z}}_k\right]
=\mathbb{P}\left[-\sum_{j=2}^N
\sum_{\mu=2}^M\xi^1_{j}\xi_1^\mu\xi_j^\mu>k-1\Big
\vert\mathcal{Z}_k\right].
\end{align*}
Conditionally on $\mathcal{Z}_k$, the distribution of $\sum_{j=2}^N
\sum_{\mu=2}^M\xi^1_{j}\xi_1^\mu\xi_j^\mu$ is symmetric and agrees with the distributions of
$
\sum_{j=2}^k
\sum_{\mu=2}^M\xi_1^\mu\xi_j^\mu\stackrel{d}{\sim}-\sum_{j=2}^k
\sum_{\mu=2}^M\xi_1^\mu\xi_j^\mu.
$
Using the exponential Chebyshev inequality and the expansion of the logarithm, we compute that for $s$ not depending on $N$
\begin{align*}
&\mathbb{P}\left[-\sum_{j=2}^N\sum_{\mu=2}^M\xi^1_{j}\xi_1^\mu\xi_j^
\mu>k-1\Big\vert\mathcal{Z}_k\right]=\mathbb{P}
\left[\sum_{j=2}
^k\sum_{\mu=2}^M\xi_1^\mu\xi_j^\mu>k-1\right]\\
\leq&\exp\left[s(-k+1)\right]\left[p\left(1-p+p\cdot\text{cosh}(s)\right)^{k-1}+1-p\right]^{M-1}\\
=&\exp\left[s(-k+1)\right]\exp\left((M-1) \log[p\left(1-p+p\cdot\text{cosh}(s)\right)^{k-1}+1-p]\right)\\
=&\exp\left[s(-k+1)\right]\exp\left((M-1) \log[1-p^2(k-1)+p^2(k-1) \cosh(s)+\mathcal{O}(p^3k^2)]\right)\\
\leq&\exp\left[-s(k-1)+(k-1)\alpha(\text{cosh}(s)-1)+\mathcal{O}(pk^2)\right].
\end{align*}
Minimizing the exponent in $s$ yields $s^*_\alpha=\text{arsinh}(1/\alpha)$ as global minimum. It satisfies
\begin{align*}
-s^*_\alpha+\alpha(\text{cosh}(s^*_\alpha)-1)=-\text{arsinh}(1/\alpha)+\alpha(\text{cosh}(\text{arsinh}(1/\alpha))-1)
\end{align*}
which is negative for every $\alpha>0$ because $g(x)=-x\text{arsinh}(x)+\text{cosh}(\text{arsinh}(x))-1$ is negative on $\mathbb{R}_+$. The probability vanishes because $k$ can be assumed to tend to infinity.

It remains to examine the probability in (\ref{BEG second type error first probability}). We show that
\begin{align}\label{BEG probability three}
&\mathbb{P}\left[\vert S_1(\xi^1)\vert=S_1(\xi^1),\text{  }S_1(\xi^1)+\theta_1(\xi^1)-\gamma\log(N)<0\Big\vert
\bar{\mathcal{Z}}_k\right]\nonumber\\
\leq&\mathbb{P}\left[S_1(\xi^1)+\theta_1(\xi^1)-
\gamma\log(N)<0\Big\vert\bar{\mathcal{Z}}_k\right]
\nonumber\\
=&\mathbb{P}\left[\sum_{j= 2}^N\sum_{\mu=1}^M\xi^1_j\xi^\mu_j\xi^\mu_1+\frac{\left(\xi_j^1\right)^2}{(1-p)^2}\eta_1^\mu\eta_j^\mu-\gamma\log(N)<0\Big\vert\bar{\mathcal{Z}}_k\right]\nonumber\\
=&\mathbb{P}\left[2(k-1)-\gamma\log(N)+\sum_{1<j\leq k}\sum_{\mu=2}^M\xi^1_j\xi^\mu_j\xi^\mu_1+\frac{\left(\xi_j^1\right)^2}{(1-p)^2}\eta_1^\mu\eta_j^\mu<0\Big\vert\bar{\mathcal{Z}_k}
\right]\nonumber\\
=&\mathbb{P}\left[-\left(\sum_{1<j\leq k}\sum_{\mu=2}^M\xi^\mu_j\xi^\mu_1+\frac{1}{(1-p)^2}\eta_1^\mu\eta_j^\mu\right)>2k-2-\gamma\log(N)\right].
\end{align}
Using the same arguments as for the previous two probabilities, especially in (\ref{exp EW BEG komplett}), for $u>0$ the
probability in (\ref{BEG probability three}) is bounded by
\begin{align}\label{exponent_prob_three}
&\mathbb{P}\left[-\left(\sum_{1<j\leq k}\sum_{\mu=2}^M\xi^\mu_j\xi^\mu_1+\frac{1}{(1-p)^2}\eta_1^\mu\eta_j^\mu\right)>2k-2-\gamma\log(N)\right]\nonumber \\
\leq&\exp\left[u\left(-2k+2+\gamma\log(N)\right)+k\alpha\left(\frac{1}{2}e^{-2u}-\frac{1}{2}+u\right)\right]\left(1+\mathcal{O}(pk^2)\right).
\end{align}
Due to the condition of the theorem, $\gamma<2$. If $\alpha>2-\gamma$, let $0<\delta<\frac{\gamma}{2-\alpha}-1$ and $\delta<\frac{2-\gamma}{2}$.
Then, for $k=\rho\log(N)\in A_\delta$  we choose
$u^*_{\rho,\alpha,\gamma}=-\frac{1}{2}\log\left(1-\frac{2}{\alpha}+\frac{\gamma}{\rho\alpha}\right)
$
to minimise the function
$u\left(-2k+\gamma\log(N)\right)+k\alpha\left(\frac{1}{2}e^{-2u}-\frac{1}{2}+u\right)
$ in $u$ (here and in the sequel we ignore the +2 in the exponent of \eqref{exponent_prob_three}, because it is negligible with respect to $\gamma \log N$).
Inserting $u^*_{\rho,\alpha,\gamma}$ and $k=\rho\log(N)$ into the above expression leads to an exponent of the form
\begin{align*}
&u^*_{\rho,\alpha,\gamma}\left(-2\rho\log(N)+\gamma\log(N)\right)+\rho\log(N)\alpha\left(\frac{1}{2}e^{-2u^*_{\rho,\alpha,\gamma}}-\frac{1}{2}+u^*_{\rho,\alpha,\gamma}\right)\\
=&\frac{1}{2}\log(N)\rho\alpha\left[-\log\left(1-\frac{2}{\alpha}+\frac{\gamma}{\rho\alpha}\right)\left(1-\frac{2}{\alpha}+\frac{\gamma}{\rho\alpha}\right)-\frac{2}{\alpha}+\frac{\gamma}{\rho\alpha}\right]\\
=&\frac{1}{2}\log(N)\rho\alpha\left[-v_{\rho,\alpha,\gamma}\log(v_{\rho,\alpha,\gamma})+v_{\rho,\alpha,\gamma}-1\right],
\end{align*}
where we set $v_{\rho, \alpha , \gamma}=1-\frac{2}{\alpha}+\frac{\gamma}{\rho\alpha}.$
Now $-x\log(x)+x-1<0$ on $\mathbb{R}_+$. Since
$2\rho-\gamma>0$ and $\alpha>2-\frac{\gamma}{\rho}$ due to the choice of $\delta$, both, $u^*_{\rho,\alpha,\gamma}>0$ and $v_{\rho,\alpha,\gamma}>0$, hold.

If $0<\alpha <2-\gamma$, we may choose any $u>0$. For $0<\delta<\frac{2-\gamma-\alpha}{2+\alpha}$ 
and on $A_\delta$,\label{page 121}
\begin{align*}
&u\left(-2k+2+\gamma\log(N)\right)+k\alpha\left(\frac{1}{2}e^{-2u}-\frac{1}{2}+u\right)\\
\leq& u\log(N)\left(-2(1-\delta)+\gamma+(1+\delta)\alpha \right)+2u+ (1-\delta)\log(N) \alpha \frac{1}{2}\left(e^{-2u}-1\right)\\
\leq&2u+ (1-\delta)\log(N) \alpha \frac{1}{2}\left(e^{-2u}-1\right)
\end{align*}
due to the choice of $\delta$. This tends to $-\infty$ for fixed $u>0$.

For $\alpha=2-\gamma$, we estimate
\begin{align*}&u\left(-2k+\gamma\log(N)\right)+k\alpha\left(\frac{1}{2}e^{-2u}-\frac{1}{2}+u\right)\\
=&u\left(-(\gamma+\alpha) k+\gamma\log(N)\right)+k\alpha\left(\frac{1}{2}e^{-2u}-\frac{1}{2}+u\right)\\
\leq&u\left(-\gamma (1-\delta)\log(N)+\gamma\log(N)\right)+(1-\delta)\log(N)\alpha\left(\frac{1}{2}e^{-2u}-\frac{1}{2}\right)
\end{align*}
and choose $u=-1/2\log(\delta\gamma/((1-\delta)\alpha))$, if $\delta<\alpha/(\alpha+\gamma)$.
 As $\mathbb{P}(A_\delta)$ tends to 1 for each $\delta>0$, this finally implies that
$$\mathbb{P}(\exists i\leq N: \xi^1_i\neq0, T_i(\xi^1)\neq\xi^1_i)\leq \mathbb{P}(A_\delta)(1+\delta)\log(N)\mathbb{P}(T_1(\xi^1)\neq\xi^1_1\vert \mathcal{Z}_k)\longrightarrow0
$$
for $0<\gamma<2 $ without a further condition on $\alpha$. This proves the first part of the theorem.

In the second part of the proof we show that the bound on $\alpha$ is sharp. It is based on the following observations:
(1)
The probability of the event $A_\delta$ converges to 1, for each $\delta>0$.

(2) The random variables
$$X_1(k)=\sum_{\mu=2}^M\mathds{1}_{\lbrace\sum_{j\leq k}\vert\xi^\mu_j\vert=1\rbrace},\quad X_2(k)=\sum_{\mu=2}^M\mathds{1}_{\lbrace\sum_{j\leq k}\vert\xi^\mu_j\vert=2\rbrace},\quad X_3(k)=\sum_{\mu=2}^M\mathds{1}_{\lbrace\sum_{j\leq k}\vert\xi^\mu_j\vert>2\rbrace}
$$
are Binomially distributed with parameters $M-1$ and
$$p_1(k)=kp(1-p)^{k-1},\quad p_2(k)=\binom{k}{2}p^2(1-p)^{k-2}\quad\text{and}\quad
p_3(k)=\binom{k}{3}p^3+\mathcal{O}(k^4p^4),
$$
respectively. Chebyshev's inequality implies for each $\delta>0$
\begin{align*}
& \mathbb{P}\left(\frac{X_1(k)}{\alpha k \frac{N}{\log(N)}}\notin(1-\delta,1+\delta)
\right)\leq \frac{\mathbb{V}(X_1(k))}{\left[\delta \alpha k N/\log N-Mk(k-1)p^2+\mathcal{O}(pk^3)\right]^2}
\end{align*}
where the $Mk(k-1)p^2+\mathcal{O}(pk^3)$ compensates for the fact that $\E X_1(k)$ is not exactly equal to $Mkp= \alpha k N /\log N$.
The right hand side can be estimated by
\begin{align*}
& \frac{\mathbb{V}(X_1(k))}{\left[\delta \alpha k N/\log N-Mk(k-1)p^2+\mathcal{O}(pk^3)\right]^2}
\leq  \frac{Mkp}{M^2k^2p^2 [\delta-(k-1)p +\mathcal{O}(p^2k^2)]^2}\\
&\leq  \frac{\log(N)}{[\delta-(k-1)p+\mathcal{O}(p^2k^2)]^2\alpha kN},
\end{align*}
which tends to 0 as $N,k$ tend to infinity for the given choice of $k\leq(1+\delta)\log(N)$.
Likewise we obtain
\begin{align*}
&\mathbb{P}\left(\frac{X_2(k)}{\alpha \binom{k}{2}}\notin(1-\delta,1+\delta)\right)
\le
\frac{1}{(\delta-p(k-2)+\alpha\binom{k}{2}\mathcal{O}(p^2k^4))^2}\longrightarrow0
\end{align*}
as $N \to \infty$. Moreover,
$\mathbb{P}\left(X_3(k)\neq0\right)\leq Mp_3(k)= M\left(\binom{k}{3}p^3+\mathcal{O}(k^4p^4)\right)\longrightarrow0
$
as $N$ and $k\leq (1+\delta ) \log(N)$ tend to infinity.
The complementary sets
$$\Bigg\lbrace\frac{X_1(k)}{\alpha k \frac{N}{\log(N)}}\in(1-\delta,1+\delta)\Bigg\rbrace,\quad \Bigg\lbrace\frac{X_2(k)}{\alpha \binom{k}{2}}\in(1-\delta,1+\delta)\Bigg\rbrace,\quad\lbrace X_3(k)=0\rbrace
$$
are denoted by $B_\delta(k)$, $C_\delta(k)$ and $D(k)$.

(3)
For each $n\in\mathbb{N}$, $n\leq M-1$ and arbitrary $i>k$, we estimate by a union bound
\begin{align*}
\mathbb{P}\left[\sum_{\mu:\sum_{j\leq k}\vert\xi^\mu_j\vert=2}\quad\sum_{j\leq k}\vert\xi^\mu_j\xi_i^\mu\vert> 0\Big\vert X_2(k)=n\right]\leq n p.
\end{align*}
This yields
\begin{align}\label{X_2 Abschaetzung Kapitel 4}
\max_{\substack{k,n\in\mathbb{N}:k/\log(N),\\n/(\alpha\binom{k}{2})\in(1-\delta,1+\delta)}}&\mathbb{P}\left[\sum_{\mu:\sum_{j\leq k}\vert\xi^\mu_j\vert=2}\sum_{j\leq k}\vert\xi^\mu_j\xi_i^\mu\vert> 0\Big\vert  X_2(k)=n\right]
\bigskip\leq
(1+\delta)^3 \alpha\frac{\log(N)^3}{2N}.
\end{align}
This means that neuron $i>k$ is not activated in any message with more than one of the activated neurons of message 1, with high probability.
The event
$\lbrace\sum_{\mu:\sum_{j\leq k}\vert\xi^\mu_j\vert=2}\sum_{j\leq k}\vert\xi^\mu_j\xi_i^\mu\vert= 0 \rbrace
$
is denoted by $C(k,i)$.

(4)
For $i>k$ let $X_4(k,i)$ be defined as
$X_4(k,i):=\sum_{\mu=2}^M\mathds{1}_{\lbrace\sum_{j\leq k}\vert\xi^\mu_j\vert=0\rbrace}\mathds{1}
_{\lbrace\xi^\mu_i\neq0\rbrace}.
$
Conditionally on $\mathcal{F}_k^N=\sigma(\xi^\mu_j,\mu\geq2, j\leq k)$ the random variable $X_4(k,i)$ is Binomially distributed with parameters $M-1-X_1(k)-X_2(k)-X_3(k)$ and $p$. The conditional probability of
$$E_\delta(k,i):=\Bigg\lbrace\frac{X_4(k,i)\log(N)}{\alpha N}\in(1-\delta,1+\delta)\Bigg\rbrace,
$$
given the intersection of the sets $B_\delta(k)$, $C_\delta(k)$, $D(k)$
is at least
\begin{align}\label{BEG E}
&\mathbb{P}\left(E_\delta(k,i)\vert B_\delta(k)\cap C_\delta(k)\cap D(k)\right)
\geq
\min_{B_\delta(k)\cap C_\delta(k)\cap D(k)}\mathbb{P}( E_\delta(k,i)\big\vert \mathcal{F}_k^N)\nonumber\\
=&\min_{\substack{m,n\in\mathbb{N}:m\log(N)/(\alpha k N),\\n/(\alpha \binom{k}{2} )\in(1-\delta,1+\delta)}}\mathbb{P}\left(E_\delta(k,i)\vert X_1(k)=m, X_2(k)=n, X_3(k)=0\right).
\end{align}
Conditionally on $\lbrace X_1(k)=m, X_2(k)=n, X_3(k)=0\rbrace$, $X_4(k,i)$ is Binomially distributed with parameters $M-1-m-n$ and $p$; applying the exponential Chebyshev inequality gives for $t>0$
 \begin{align}\label{BEG X4 E(k,i) Teil 1}
&1-\max_{\substack{m,n\in\mathbb{N}:m\log(N)/(\alpha k N),\\n/(\alpha \binom{k}{2} )\in(1-\delta,1+\delta)}}\mathbb{P}\left[\frac{X_4(k,i)}{Mp}\geq1+\delta\Big\vert X_1(k)=m, X_2(k)=n, X_3(k)=0\right]\nonumber\\
\geq&1-\max_{\substack{m,n\in\mathbb{N}:m\log(N)/(\alpha k N),\\n/(\alpha \binom{k}{2} )\in(1-\delta,1+\delta)}}\exp[-t(1+\delta)Mp](1-p+pe^t)^{M-1-m-n}\nonumber\\
\geq&1-\exp[-t(1+\delta)Mp](1-p+pe^t)^{M}\geq1-\exp[(-\log(1+\delta)(1+\delta)+\delta) Mp],
\end{align}
as well as
 \begin{align}\label{BEG X4 E(k,i) Teil 2}
&1-\max_{\substack{m,n\in\mathbb{N}:m\log(N)/(\alpha k N),\\n/(\alpha \binom{k}{2} )\in(1-\delta,1+\delta)}}\mathbb{P}\left[\frac{X_4(k,i)}{Mp}\leq1-\delta\Big\vert X_1(k)=m, X_2(k)=n, X_3(k)=0\right]\nonumber\\
\geq&1-\exp[t(1-\delta)Mp](1-p+pe^{-t})^{M-1-(1+\delta)\alpha\left( Nk/\log(N)+\binom{k}{2}\right)}\nonumber\\
\geq&1-\exp[(-\log(1-\delta)(1-\delta)-\delta) Mp]\exp\left[\delta p+ \delta(1+\delta)\alpha \left(k+\binom{k}{2} p\right)\right].
\end{align}
In particular, $N\cdot \mathbb{P}(E_\delta(k,i)\vert B_\delta(k)\cap C_\delta(k)\cap D(k))\longrightarrow0$.

(5)
As penultimate point, we observe that $X_5(k,i):=\sum_{\mu:\sum_{j\leq k}\vert\xi^\mu_j\vert=1}\vert\xi^\mu_i\vert,
$
is Binomially distributed with parameters $p$ and $X_1(k)$ conditionally on $\mathcal{F}_k^N$. Assuming that $k\leq (1+\delta) \log(N)$ and that $X_1(k)\leq (1+\delta)\alpha (kN/\log(N))$, the random variable $X_5(k,i)$ is asymptotically Poisson distributed with Parameter $\lambda=pX_1(k)$; we denote this distribution by $\pi_\lambda$. Moreover, by \cite{LeCam} the total variation distance of the latter two distributions is for some $i>k$ at most
$$\sum_{m=0}^\infty\Bigg\vert \mathbb{P}\left(\sum_{\mu:\sum_{j\leq k}\vert \xi^\mu_j\vert =1}\vert\xi^\mu_i\vert=m \Big\vert \mathcal{F}_k^N\right) - \pi_{pX_1(k)}(m)\Bigg\vert\leq2p^2X_1(k).
$$

(6)
Finally, the events
$\Bigg\lbrace \sum_{\mu=2}^M\sum_{j\leq k}\xi^\mu_i\xi^\mu_j+\frac{1}{(1-p)^2}\eta_i^\mu\eta_j^\mu< \gamma\log(N)\Bigg\rbrace,\quad i>k,
$
are conditionally independent, given $\mathcal{F}_k^N$.

Now we are ready to prove the desired statement. First, we estimate
\begin{align}\label{BEG k upper bound}
&\mathbb{P}\left(\exists i\leq N: T_i(\xi^1)\neq0\right)\geq\mathbb{P}\left(\exists i\leq N:\xi^1_i=0, T_i(\xi^1)\neq0\right)\nonumber\\
\geq &\mathbb{P}(A_\delta)\min_{\substack{k\in\mathbb{N}:k/\log(N)\\\in(1-\delta,1+\delta)}}\mathbb{P}\left(\exists i\leq N:\xi^1_i=0, T_i(\xi^1)\neq0\vert \mathcal{Z}_k\right)\nonumber\\
=&\mathbb{P}(A_\delta)\min_{\substack{k\in\mathbb{N}:k/\log(N)\\\in(1-\delta,1+\delta)}}\mathbb{P}\left(\exists i>k:\sum_{\mu=2}^M\sum_{j\leq N}\xi^1_j\xi^\mu_i\xi^\mu_j+\frac{1}{(1-p)^2}\left(\xi^1_j\right)^2\eta_i^\mu\eta_j^\mu\geq\gamma\log(N)\Big\vert \mathcal{Z}_k\right)\nonumber\\
=&\mathbb{P}(A_\delta)\min_{\substack{k\in\mathbb{N}:k/\log(N)\\\in(1-\delta,1+\delta)}}\mathbb{P}\left(\exists i>k:\sum_{\mu=2}^M\sum_{j\leq k}\xi^\mu_i\xi^\mu_j+\frac{1}{(1-p)^2}\eta_i^\mu\eta_j^\mu\geq\gamma\log(N)\right);
\end{align}
because conditionally on $\mathcal{Z}_k$,
$$\left(\sum_{\mu=2}^M\sum_{j\leq N}\xi^1_j\xi^\mu_i\xi^\mu_j+\frac{1}{(1-p)^2}\left(\xi^1_j\right)^2\eta_i^\mu\eta_j^\mu,\thickspace i>k\right)\sim\left(\sum_{\mu=2}^M\sum_{j\leq k}\xi^\mu_i\xi^\mu_j+\frac{1}{(1-p)^2}\eta_i^\mu\eta_j^\mu,\thickspace i>k\right),
$$
as we saw in the first part of the proof.

Writing
$B_\delta\cap C_\delta\cap D(k):=B_\delta(k)\cap C_\delta(k)\cap D(k)
$
we can thus conclude
\begin{align}\label{BEG bedingt}
&\mathbb{P}\left(\exists i>k:\sum_{\mu=2}^M\sum_{j\leq k}\xi^\mu_i\xi^\mu_j+\frac{1}{(1-p)^2}\eta_i^\mu\eta_j^\mu\geq\gamma\log(N)\right)\nonumber\\
=&1-\mathbb{E}_{(\xi^\mu_j,j\leq k, \mu\geq 2)}\left[\mathbb{P}\left(\sum_{\mu=2}^M\sum_{j\leq k}\xi^\mu_N\xi^\mu_j+\frac{1}{(1-p)^2}\eta_N^\mu\eta_j^\mu<\gamma \log(N)\Big\vert \mathcal{F}_k^N\right)^{N-k}\right]\nonumber\\
\geq&1-\max_{\substack{B_\delta\cap C_\delta\\\cap D(k)}}\mathbb{P}\left(\sum_{\mu=2}^M\sum_{j\leq k}\xi^\mu_N\xi^\mu_j+\frac{1}{(1-p)^2}\eta_N^\mu\eta_j^\mu<\gamma \log(N)\Big\vert \mathcal{F}_k^N\right)^{N-k}-\mathbb{P}[(B_\delta\cap C_\delta\cap D(k))^c].
\end{align}
For every $k$ in the considered set, the probability of $B_\delta\cap C_\delta\cap D(k)$
tends to 1 and it suffices to analyse the other probability in the last line. We rewrite the sum
\begin{align*}
&\sum_{\mu=2}^M\sum_{j\leq k}\xi^\mu_N\xi^\mu_j+\frac{1}{(1-p)^2}\eta_N^\mu\eta_j^\mu\\
&=\sum_{\substack{\mu:\sum_{j\leq k}\vert\xi^\mu_j\vert=0}}\thickspace\sum_{j\leq k}\xi^\mu_N\xi^\mu_j+\frac{1}{(1-p)^2}\eta_N^\mu\eta_j^\mu+\sum_{\substack{\mu:\sum_{j\leq k}\vert\xi^\mu_j\vert=1}}\thickspace\sum_{j\leq k}\xi^\mu_N\xi^\mu_j+\frac{1}{(1-p)^2}\eta_N^\mu\eta_j^\mu\\
+&\sum_{\substack{\mu:\sum_{j\leq k}\vert\xi^\mu_j\vert=2}}\thickspace\sum_{j\leq k}\xi^\mu_N\xi^\mu_j+\frac{1}{(1-p)^2}\eta_N^\mu\eta_j^\mu+\sum_{\substack{\mu>2:\sum_{j\leq k}\vert\xi^\mu_j\vert>2}}\thickspace\sum_{j\leq k}\xi^\mu_N\xi^\mu_j+\frac{1}{(1-p)^2}\eta_N^\mu\eta_j^\mu.
\end{align*}
On $B_\delta(k)\cap C_\delta(k)\cap C(k,i)\cap D(k)\cap E_\delta (k,i)$, the last sum is zero; the penultimate sum is
\begin{align}\label{BEG Summe 2}
&\sum_{\substack{\mu:\sum_{j\leq k}\vert\xi^\mu_j\vert=2}}\thickspace\sum_{j\leq k}\xi^\mu_N\xi^\mu_j+\frac{1}{(1-p)^2}\eta_N^\mu\eta_j^\mu=\left(-\frac{2p}{1-p}+(k-2)\frac{p^2}{(1-p)^2}\right) X_2(k)
\end{align}
and therefore tends to 0 on this set, because $X_2(k)/(\alpha \binom{k}{2})\in(1-\delta,1+\delta)$.

For the first sum, note that
$\sum_{\mu=2}^M\mathds{1}_{\sum_{j\leq k}\vert\xi^\mu_j\vert=0}=M-1-X_1(k)-X_2(k)-X_3(k)
$
and thus
\begin{align}\label{BEG Summe null}
&\sum_{\substack{\mu:\sum_{j\leq k}\vert\xi^\mu_j\vert=0}}\thickspace\sum_{j\leq k}\xi^\mu_N\xi^\mu_j+\frac{\eta_N^\mu\eta_j^\mu}{(1-p)^2}\nonumber\\
=&\left[M-1-X_1(k)-X_2(k)-X_3(k)-X_4(k,N)\right]k\frac{p^2}{(1-p)^2}-X_4(k,N)k\frac{p}{1-p}\nonumber\\
\geq&\frac{Mkp^2}{(1-p)^2}-X_4(k,N)k\frac{p}{1-p}\geq\alpha k \frac{1}{(1-p)^2}-\alpha k \frac{1+\delta}{1-p}
\end{align}
on the set we consider.

The remaining second sum is equal to \begin{align}\label{BEG Summe Eins}
&\sum_{\substack{\mu:\sum_{j\leq k}\vert\xi^\mu_j\vert=1}}\thickspace\sum_{j\leq k}\xi^\mu_N\xi^\mu_j+\frac{\eta_N^\mu\eta_j^\mu}{(1-p)^2}
=
\sum_{\substack{\mu:\sum_{j\leq k}\vert\xi^\mu_j\vert\\=1,\vert\xi^\mu_N\vert=1}}(1-p\frac{k-1}{1-p}\thinspace+\sum_{j\leq k}\xi^\mu_N\xi^\mu_j)+\sum_{\substack{\mu:\sum_{j\leq k}\vert\xi^\mu_j\vert\\=1,\vert \xi^\mu_N\vert=0}}\sum_{j\leq k}\frac{\eta_N^\mu\eta_j^\mu}{(1-p)^2}.
\end{align}
The right hand side of the last line in (\ref{BEG Summe Eins}) is at least
\begin{align}\label{BEG Summe Eins Teil 1}
\sum_{\substack{\mu:\sum_{j\leq k}\vert\xi^\mu_j\vert\\=1,\vert \xi^\mu_N\vert=0}}\thickspace\sum_{j\leq k}\frac{\eta_N^\mu\eta_j^\mu}{(1-p)^2}=\left(X_1(k)-X_5(k,N)\right)\left((k-1)\frac{p^2}{(1-p)^2}-\frac{p}{1-p}\right)\geq X_1(k)\frac{-p}{1-p}
\end{align}
because $X_1(k)\geq X_5(k,N).$
 It remains to examine the left hand side of the last line in (\ref{BEG Summe Eins}); this is the last step in the proof.

Due to (\ref{BEG k upper bound}) and (\ref{BEG bedingt}), our goal is to show that
\begin{align*}
\lim_{N\rightarrow\infty}\max_{\substack{k/\log(N)\in\\(1-\delta,1+\delta)}}\max_{B_\delta\cap C_\delta\cap D(k)}\left[1-\mathbb{P}\left(\sum_{\mu=2}^M\sum_{j\leq k}\xi^\mu_N\xi^\mu_j+\frac{1}{(1-p)^2}\eta_N^\mu\eta^\mu_j\geq \gamma\log(N)\big\vert\mathcal{F}_k^N\right)\right]^{N-k}=0.
\end{align*}
This is fulfilled if
\begin{align}\label{Log BEG Bedingung}
&\lim_{N\rightarrow\infty}\min_{\substack{k/\log(N)\in\\(1-\delta,1+\delta)}}\min_{B_\delta\cap C_\delta\cap D(k)}\frac{\log\left[\mathbb{P}\left(\sum_{\mu=2}^M\sum_{j\leq k}\xi^\mu_N\xi^\mu_j+\frac{1}{(1-p)^2}\eta_N^\mu\eta^\mu_j\geq \gamma\log(N)\big\vert\mathcal{F}_k^N\right)\right]}{\log(N)}>-1.
\end{align}
Due to (\ref{BEG Summe 2}), (\ref{BEG Summe null}), (\ref{BEG Summe Eins}) and (\ref{BEG Summe Eins Teil 1}), on $B_\delta(k)\cap C_\delta(k)\cap C(k,i)\cap D(k)\cap E_\delta (k,i)$, the sum under consideration is at least
\begin{align}\label{BEG auf Schnitt}
&\sum_{\mu=2}^M\sum_{j\leq k}\xi^\mu_N\xi^\mu_j+\frac{\eta_N^\mu\eta^\mu_j}{(1-p)^2}
\geq \left(-\frac{2p}{1-p}+(k-2)\frac{p^2}{(1-p)^2}\right) (1-\delta)\alpha \binom{k}{2}+\alpha k\frac{1}{(1-p)^2}\nonumber\\
&-\alpha k \frac{1+\delta}{1-p}-(1+\delta) \alpha k \frac{N}{\log(N)} \frac{p}{1-p}+\sum_{\substack{\mu:\sum_{j\leq k}\vert\xi^\mu_j\vert\\=1,\vert\xi^\mu_N\vert=1}}\left(1-p\frac{k-1}{1-p}\thinspace+\sum_{j\leq k}\xi^\mu_N\xi^\mu_j\right)\nonumber\\
\geq& \alpha k \frac{-1-2\delta}{1-p}+\mathcal{O}(k^2p)+
\sum_{\substack{\mu:\sum_{j\leq k}\vert\xi^\mu_j\vert\\=1,\vert\xi^\mu_N\vert=1}}\left(1-p\frac{k-1}{1-p}\thinspace+\sum_{j\leq k}\xi^\mu_N\xi^\mu_j\right).
\end{align}
We have seen in 4., (\ref{BEG E}), (\ref{BEG X4 E(k,i) Teil 1}) and (\ref{BEG X4 E(k,i) Teil 2}), using $\varepsilon_\delta=\min(\log(1+\delta)(1+\delta)+\delta;\log(1-\delta)(1-\delta)-\delta)>0$, that
$$\max_{B_\delta(k)\cap C_\delta(k)\cap D(k)}\mathbb{P}(E_\delta(k,N)^c \vert \mathcal{F}_k^N)\leq\exp\left[-\varepsilon_\delta \alpha \frac{N}{\log(N)}+\mathcal{O}(k)\right]
$$
and in (\ref{X_2 Abschaetzung Kapitel 4}), that
$$\max_{B_\delta(k)\cap C_\delta(k)\cap D(k)}\mathbb{P}(C(k,N)^c \vert \mathcal{F}_k^N)\leq (1+\delta)^3 \alpha \frac{\log(N)^3}{N}.
$$
So we can conclude for the probability in (\ref{Log BEG Bedingung}), using additionally (\ref{BEG auf Schnitt}),
\begin{align}\label{BEG P Rest}
&\min_{B_\delta\cap C_\delta\cap D(k)}\mathbb{P}\left(\sum_{\mu=2}^M\sum_{j\leq k}\xi^\mu_N\xi^\mu_j+\frac{1}{(1-p)^2}\eta_N^\mu\eta^\mu_j\geq \gamma\log(N)\big\vert\mathcal{F}_k^N\right)\nonumber\\
\geq&\min_{B_\delta(k)}\mathbb{P}\Bigg(
\sum_{\substack{\mu:\sum_{j\leq k}\vert\xi^\mu_j\vert\\=1,\vert\xi^\mu_N\vert=1}}\left(1-p\frac{k-1}{1-p}\thinspace+\sum_{j\leq k}\xi^\mu_N\xi^\mu_j\right)\geq \gamma\log(N)+\frac{\alpha k (1+2\delta)}{1-p}+\mathcal{O}(k^2p)\Big\vert\mathcal{F}_k^N\Bigg)\nonumber\\
&\quad\quad\quad\quad-\exp\left[-\varepsilon_\delta \alpha \frac{N}{\log(N)}+\mathcal{O}(k)\right]-(1+\delta)^3 \alpha \frac{\log(N)^3}{N}.
\end{align}
Finally we consider the behaviour of the summand on the left hand side in the last line of (\ref{BEG Summe Eins}), that is,
\begin{align}\label{BEG Summe Eins Teil 2}
\sum_{\substack{\mu:\sum_{j\leq k}\vert\xi^\mu_j\vert\\=1,\vert\xi^\mu_N\vert=1}}\left(1-p\frac{k-1}{1-p}\thinspace+\sum_{j\leq k}\xi^\mu_N\xi^\mu_j\right).
\end{align}
Conditionally on $\xi^\mu_j,j\leq k,\mu\geq 2$ and $\vert \xi^\mu_N\vert,\mu\geq 2$, it is distributed as a sum of $X_5(k,N)$ independent and identically distributed random variables $Z_n(p,k),$ $n\geq 1$, with distribution
$\mathbb{P}\left(Z_n(p,k)=-p\frac{k-1}{1-p} \right)=\mathbb{P}\left(Z_n(p,k)=2-p\frac{k-1}{1-p}\right)=\frac{1}{2}.
$
$X_5(k,N)$ has still to be determined. Given $X_1(k)$, it is Binomially distributed with parameters $X_1(k)$ and $p$. Consequently, the sum in (\ref{BEG Summe Eins Teil 2}) is, given $X_1(k)$, distributed as a random sum of random variables $Z_n(p,k)$, $n\geq 1$, and length determined by a $\mathrm{Bin}(X_1(k),p$)-distributed random variable $R_{X_1(k)}$, such that $R_{X_1(k)},Z_1(p,k),\ldots$ are independent. This Binomial distribution can be approximated by a Poisson distribution; the subsequent computations are therefore made for a Poisson distribution, instead.
For each $\varepsilon\geq0$ let $Z_n(\varepsilon)$, $n\geq 1$, be independent and identically distributed such that
$$\mathbb{P}(Z_n(\varepsilon)=-\varepsilon)=\mathbb{P}(Z_n(\varepsilon)=2-\varepsilon)=\frac{1}{2}.
$$
Let now $Y_{\lambda}$ denote a Poisson random variable with parameter $\lambda$. For independent sets of random variables $(Y_{\lambda k},Z_n(\varepsilon),n\geq1)$, respectively $(Y_{\lambda}^{r},Z_{n}^r(\varepsilon),r=1,\ldots,k,n\geq1)$,  with $Y_{\lambda}^{r}\sim\mathrm{Poi}(\lambda)$, $Z_n^r(\varepsilon)\sim Z_1(\varepsilon)$, for each $r$ and $n\geq1$, $\sum_{n=1}^{Y_{\lambda k}}Z_n(\varepsilon)\mbox{ and }\sum_{r=1}^k\sum_{n=1}^{Y_{\lambda }^{r}}Z_{n}^r(\varepsilon).
$ have the same distribution.
The $\sum_{n=1}^{Y_{\lambda}^{r}}Z_{n}^r(\varepsilon)$, $1\leq r\leq k$, then are independent. \\By Cram\'er's theorem (see \cite{dembozeitouni}, Theorem 2.2.3), if
$x>\mathbb{E}\left(\sum_{n=1}^{Y_{\lambda}^1}Z_{n}^1(\varepsilon)\right)$ and $ \mathbb{E}\left(\exp t \sum_{n=1}^{Y_\lambda^1}Z_n^1(\varepsilon)\right)<\infty, t\in\mathbb{R}
$
$$\lim_{k\rightarrow\infty}\frac{1}{k}\log\left(\mathbb{P}\left(\sum_{n=1}^{Y_{\lambda k}}Z_n(\varepsilon)\geq x k\right)\right)=\lim_{k\rightarrow\infty}\frac{1}{k}\log\left(\mathbb{P}\left(\sum_{r=1}^k\sum_{n=1}^{Y_{\lambda}^{r}}Z_{n}^r(\varepsilon)\geq x k\right)\right)=-\Lambda_{\lambda,\varepsilon}^*(x)
$$
with
$$\Lambda^*_{\lambda,\varepsilon}(x)=\sup_{t\in \mathbb{R}}\thickspace tx-\log\left[\mathbb{E}\left(\exp t \sum_{n=1}^{Y_\lambda^1}Z_n^1(\varepsilon)\right)\right].
$$
By Wald's identity, the expectation of this sum is
$$\mathbb{E}\left(\sum_{n=1}^{Y_{\lambda}^1}Z_{n}^1(\varepsilon)\right)=\mathbb{E}\left(Y_{\lambda}^1\right)\mathbb{E}(Z_1^1(\varepsilon))=\lambda (1-\varepsilon).
$$
To compute $\Lambda^*_{\lambda,\varepsilon}$, consider the moment generating function of $\sum_{n=1}^{Y_{\lambda}^1}Z_{n}^1(\varepsilon)$:
\begin{align*}
\Lambda_{\lambda,\varepsilon}(t):=&\mathbb{E}\left(\exp t \sum_{n=1}^{Y_{\lambda}^1}Z_{n}^1(\varepsilon)\right)=\sum_{m=0}^\infty
e^{-\lambda}\frac{\lambda^m}{m!}\frac{1}{2^m}e^{-\varepsilon tm}\left(1+e^{2t}\right)^m
=
e^{- \lambda+ \frac{\lambda}{2}e^{-\varepsilon t}(e^{2t}+1)}.
\end{align*}
Then
$$\Lambda^*_{\lambda,\varepsilon}(x)=\sup_{t\in\mathbb{R}}\thickspace \left(t x-\log\left[
\mathbb{E}\left(\exp t
\sum_{n=1}^{Y_{\lambda}^1}Z_{ n}^1(\varepsilon)\right)\right]\right)=\sup_{t\in\mathbb{R}}\thickspace \left(t x+ \lambda- \frac{\lambda}{2}e^{-\varepsilon t}(e^{2t}+1)\right).
$$
For fixed $x>\lambda$ we are interested in $\lim_{\varepsilon\searrow0}\Lambda^*_{\lambda,\varepsilon}(x).$ Observe that for $t\leq 0$,
$$t x+ \lambda- \frac{\lambda}{2}e^{-\varepsilon t}(e^{2t}+1)\leq t x+ \lambda- \frac{\lambda}{2}(e^{2t}+1)=tx-\log\left(\Lambda_{\lambda,0}(t)\right).
$$
For fixed $\lambda$ and $x>\lambda$ define the continuous and differentiable functions $\psi_\varepsilon:\mathbb{R}\rightarrow\mathbb{R}$,
$$\psi_\varepsilon(t):=t x+ \lambda- \frac{\lambda}{2}e^{-\varepsilon t}(e^{2t}+1).
$$ We easily compute for $0<\varepsilon<2$
$$\psi'_\varepsilon(t)= x+\varepsilon\frac{\lambda}{2}e^{-\varepsilon t}-(2-\varepsilon)\frac{\lambda}{2}e^{(2-\varepsilon)t};\quad\lim_{t\rightarrow\infty}\psi'_\varepsilon(t)=-\infty;\quad\lim_{t\rightarrow-\infty}\psi'_\varepsilon(t)=\infty.
$$
Moreover, for this choice of $\varepsilon$,
$\psi''_\varepsilon(t)<0$ on $\mathbb{R}$. Hence
$\sup_{t\in\mathbb{R}}\psi_\varepsilon(t)=\psi_\varepsilon(t_\varepsilon),
$
with the unique root $t_\varepsilon$ of $\psi_\varepsilon'(t)$. For each $t\leq0$, $\varepsilon\geq0$ we have
$$\psi'_\varepsilon(t)\geq x+\varepsilon \frac{\lambda}{2}-(2-\varepsilon)\frac{\lambda}{2} \geq x-\lambda>0
$$
and for each $t>\log(1+2x/\lambda)$ and $0\leq \varepsilon\leq1$ we have
$\psi'_\varepsilon(t)\leq x+\frac{\lambda}{2}-\frac{\lambda}{2}e^{t}<0.$
So for $x>\lambda$, $0\leq \varepsilon\leq1$ we obtain
$\sup_{t\in\mathbb{R}}\psi_\varepsilon(t)=\sup_{t\in[0,\log(1+2x/\lambda)]}\psi_\varepsilon(t).
$
On $[0,\log(1+2x/\lambda)]$, $\psi_\varepsilon$ converges uniformly to $\psi_0$ as $\varepsilon\rightarrow0$. So we can conclude that
$\lim_{\varepsilon\searrow0}t_\varepsilon=t_0$ and $\lim_{\varepsilon\searrow0}\psi_\varepsilon(t_\varepsilon)=\psi_0(t_0).$
We finally compute  $\Lambda^*_{\lambda,0}$:
\begin{align*}
\Lambda^*_{\lambda,0}(x)=&\sup_{t\in\mathbb{R}}\thickspace t x-\log\Bigl[
\mathbb{E}\bigl(\exp t
\sum_{n=1}^{Y_{\lambda}^1}Z_{ n}^1(0)\bigr)\Bigr]
=\frac{1}{2}\log\bigl(\frac{x}{\lambda}\bigr)x-\frac{\lambda}{2}\bigl(\frac{x}{\lambda}-1\bigr),
\end{align*}
using $t_0=\frac{1}{2}\log\left(\frac{x}{\lambda}\right)$.

Altogether, we deduce for $\varepsilon>\frac{p(k-1)}{(1-p)}$, using in the first step (5), and that $\alpha k (1+2\delta)/(1-p)=\alpha k (1+2\delta)+\mathcal{O}(kp)$,
\begin{align*}
&\min_{B_\delta(k)}\mathbb{P}\Bigg(
\sum_{\substack{\mu:\sum_{j\leq k}\vert\xi^\mu_j\vert\\=1,\vert\xi^\mu_N\vert=1}}\left(1-p\frac{k-1}{1-p}\thinspace+\sum_{j\leq k}\xi^\mu_N\xi^\mu_j\right)\geq \gamma\log(N)+\frac{\alpha k (1+2\delta)}{1-p}+\mathcal{O}(k^2p)\Big\vert\mathcal{F}_k^N\Bigg)\\
\geq&\min_{B_\delta(k)}\mathbb{P}\left(
\sum_{n=1}^{Y_{p\cdot X_1(k)}}Z_n(k,p)\geq \gamma\log(N)+\alpha k (1+2\delta)+\mathcal{O}(k^2p)\Big\vert\mathcal{F}_k^N\right)-2p^2 X_1(k)\\
\geq&\min_{\rho\in(1-\delta,1+\delta)}\mathbb{P}\left(
\sum_{n=1}^{Y_{\alpha \rho k}}Z_n(\varepsilon)\geq \gamma\log(N)+\alpha k (1+2\delta)+\mathcal{O}(k^2p)\right)-2p\alpha \rho k.
\end{align*}
As the second summand is negligible in view of  \eqref{Log BEG Bedingung} and
(\ref{BEG P Rest}), it suffices to show
$$\liminf_{N\rightarrow\infty}\min_{\substack{k\in\mathbb{N}:k/\log(N)\\\in(1-\delta,1+\delta)}}\min_{\substack{\rho\in(1-\delta,\\1+\delta)}}\frac{\log \left[\mathbb{P}\left(
\sum_{n=1}^{Y_{\alpha \rho k}}Z_n(\varepsilon)\geq \gamma\log(N)+\alpha k (1+2\delta)+\mathcal{O}(k^2p)\right)\right]}{\log(N)}>-1
$$
to obtain convergence of the probability of instability of $\xi^1$ to 1. In our analysis of $\sum_{n=1}^{Y_{\lambda k}}Z_n(\varepsilon)$ we saw that
\begin{align*}
&\lim_{k\rightarrow\infty}\frac{\log\left(\mathbb{P}\left[\sum_{n=1}^{Y_{\rho \alpha k}}Z_n(\varepsilon)\geq(\frac{\gamma}{1-\delta}+ \alpha (1+2\delta))k \right]\right)}{k}
=-\Lambda^*_{\rho \alpha, \varepsilon}\left(\frac{\gamma}{1-\delta}+ \alpha (1+2\delta)\right)
\end{align*}
if $\frac{\gamma}{1-\delta}+\alpha (1+2\delta)>\rho\alpha(1-\varepsilon)$. We saw that the maximal argument $t_\varepsilon$ of $\psi_\varepsilon$ was positive; one easily sees that $tx-\log(\Lambda_{\lambda,\varepsilon}(t))$ is decreasing in $\lambda$ for fixed $\varepsilon<1/2$ and $t>0$, and therefore that
\begin{align*}
&\min_{\substack{\rho\in(1-\delta,\\1+\delta)}}\thinspace\liminf_{N\rightarrow\infty}\thinspace\min_{\substack{k\in\mathbb{N}:k/\log(N)\\\in(1-\delta,1+\delta)}}\frac{\log \left[\mathbb{P}\left(
\sum_{n=1}^{Y_{\alpha \rho k}}Z_n(\varepsilon)\geq \gamma\log(N)+\alpha k (1+2\delta)\right)\right]}{\log(N)}\\
\end{align*}
\begin{align*}
\geq &(1-\delta)\min_{\substack{\rho\in(1-\delta,\\1+\delta)}}\thinspace\liminf_{N\rightarrow\infty}\thinspace\min_{\substack{k\in\mathbb{N}:k/\log(N)\\\in(1-\delta,1+\delta)}}\frac{\log \left[\mathbb{P}\left(
\sum_{n=1}^{Y_{\alpha \rho k}}Z_n(\varepsilon)\geq \frac{\gamma}{1-\delta}k+\alpha k (1+2\delta)\right)\right]}{k}\\
\geq &(1-\delta)\left(-\Lambda^*_{(1-\delta) \alpha ,\varepsilon}\left(\frac{\gamma}{1-\delta}+ \alpha (1+2\delta)\right)\right).
\end{align*}
Due to our considerations on $Z_n(\varepsilon)$, there is some $\varepsilon'>0$ such that for all $0\leq \varepsilon<\varepsilon'$,
$$-(1-\delta)\Lambda^*_{(1-\delta) \alpha ,\varepsilon}\left(\frac{\gamma}{1-\delta}+ \alpha (1+2\delta)\right)>-1,
 \mbox{ if }
-(1-\delta)\Lambda^*_{(1-\delta) \alpha ,0}\left(\frac{\gamma}{1-\delta}+ \alpha (1+2\delta)\right)>-1.
$$
Since moreover $x \mapsto \Lambda^*_{(1-\delta) \alpha ,0}(x)$ is continuous,
this suffices to guarantee the convergence of $\mathbb{P}(\exists i: T_i(\xi^1) \neq \xi_i^1)$ to 1 (note that the term $\mathcal{O}(k^2p)$ in $$\mathbb{P}\left(
\sum_{n=1}^{Y_{\alpha \rho k}}Z_n(\varepsilon)\geq \gamma\log(N)+\alpha k (1+2\delta)+\mathcal{O}(k^2p)\right)$$ is neglibile).

Using the function $f_{\alpha,  (1-\delta)^2}(x)=1+\frac{1}{2}(1-\delta)^2\alpha (-x\log(x)+x-1)$ of the first part of the proof and $\bar{x}_{\alpha, \gamma, \delta}:=\frac{\gamma}{\alpha (1-\delta)^2}+\frac{1+2\delta}{1-\delta},$  a sufficient condition for $\mathbb{P}(\exists i: T_i(\xi^1)\neq \xi^1_i) \to 1$ is
$f_{\alpha, (1-\delta)^2}(\bar{x}_{\alpha, \gamma, \delta})>0.
$
Finally, if $\alpha$ and $\gamma$ fulfill
\begin{align}\label{upper bound Kapitel 4 Funktion}
f_{\alpha, 1}(\bar{x}_{\alpha, \gamma, 0})=1+\frac{1}{2}\alpha\left[-\log\left(1+\frac{\gamma}{\alpha}\right)(1+\gamma/\alpha)+\frac{\gamma}{\alpha}\right]>0,
\end{align}
$\delta>0$ can be chosen small enough to obtain
$f_{\alpha, (1-\delta)^2}(\bar{x}_{\alpha, \gamma,\delta})>0$. As on page \pageref{page 118}, we observe that condition (\ref{upper bound Kapitel 4 Funktion}) is fulfilled if $g_\gamma(\bar{x}_{\alpha, \gamma,0})>0$, with $\bar{x}_{\alpha, \gamma,0}=1+\frac{\gamma}{\alpha}$ and $g_\gamma(x)=x\left(1+\frac{2}{\gamma}-\log(x)\right)-1-\frac{2}{\gamma}$. Since $g_\gamma$ is positive on $(1,x^*_\gamma)$, with the root $x_\gamma^*$ of $g_\gamma$ in $(1,\infty)$, the precedent condition holds if
$\bar{x}_{\alpha, \gamma,0}\in(1,x^*_\gamma).
$
As $\bar{x}_{\alpha, \gamma,0}>1$, this is true if
$ \alpha > \frac{\gamma}{x_\gamma^*-1}.$
We conclude that $\lim_{N\rightarrow\infty}\mathbb{P}(\exists i\leq N:T_i(\xi^1)\neq \xi_i^1)=1$ holds if $\alpha>\frac{\gamma}{x^*_\gamma-1},$ as stated in the theorem.
\end{proof}

To find the optimal capacity we show:
\begin{proposition}\label{Remark BEG Gamma}
The threshold variable $\gamma$ cannot be chosen from $(2,\infty)$.
\end{proposition}

\begin{proof}
Suppose that $\gamma>2$. Choosing some $\delta< \frac{1}{2}(\gamma-2)$, we will see that,
$$\lim_{N\rightarrow\infty}\thinspace\max_{\substack{k\in\mathbb{N}:k/\log(N)\in\\(1-\delta,1+\delta)}}\mathbb{P}\left(T_1(\xi^1)\neq\xi^1_1\vert \mathcal{Z}_k\right)=1,
$$
independently of the choice of $\alpha$. Using (\ref{exp EW BEG komplett}), we easily conclude
\begin{align*}
&\max_{\substack{k\in\mathbb{N}:k/\log(N)\in\\(1-\delta,1+\delta)}}\mathbb{P}\left(T_1(\xi^1)=\xi^1_1\vert \mathcal{Z}_k\right)
\leq\max_{\substack{k\in\mathbb{N}:k/\log(N)\in\\(1-\delta,1+\delta)}}\mathbb{P}\left(\vert S_1(\xi^1)\vert+ \theta_1(\xi^1) \geq\gamma\log(N)\vert\mathcal{Z}_k\right)\\
\leq& \max_{\substack{k\in\mathbb{N}:k/\log(N)\in\\(1-\delta,1+\delta)}}\mathbb{P}\left[\left(\sum_{1<j\leq k}\sum_{\mu=2}^M\xi^\mu_1\xi^\mu_j+\frac{1}{(1-p)^2}\eta_1^\mu\eta_j^\mu\right)\geq\gamma\log(N)+2-2k\right]\\
\leq&\exp\left[-t\log(N)(\gamma-2-2\delta)\right]\exp\left[(1+\delta)\log(N)\alpha\left(\frac{1}{2} e^{2t}-\frac{1}{2}-t\right)+\mathcal{O}\left(p\log(N)^2\right)\right]\\
\leq&\exp\left[\log(N)\left((1+\delta)\alpha \frac{1}{2}\left(-w_{\gamma,\delta,\alpha}\log(w_{\gamma,\delta,\alpha})+w_{\gamma,\delta,\alpha}-1\right)\right)\right],
\end{align*}
with $w_{\gamma,\delta,\alpha}:=1+\frac{\gamma-2(1+\delta)}{(1+\delta)\alpha}$ and
after having used $t=\frac{1}{2}\log\left(w_{\gamma,\delta,\alpha}\right)>0$. The probability tends to 0.
\end{proof}
As a result we see that the critical value $\gamma^*=\sup \lbrace \gamma>0: \gamma\text{ is an admissible threshold}\rbrace=2$.
The corresponding capacity is immediately obtained as.
\begin{proposition}
With $\gamma^*=2$ the critical value $$\alpha^*=\sup\lbrace \alpha >0:  \alpha \text{ is an admissible capacity}\rbrace$$ is $\alpha^*=\frac{2}{x_2^*-1}\approx0.51.$
Here $x_2^*$ is the root of
$g_2(x)=x\left(2-\log(x)\right)-2$
in $(1,\infty)$.
%
\end{proposition}

\begin{rem}
The storage capacity thus obtained shows that our new version of the BEG model outperforms the models discussed in \cite{GHLV16}.
We illustrate the corresponding capacities in Figure \ref{fig1}.
\end{rem}
\begin{figure}
\includegraphics[scale=0.2]{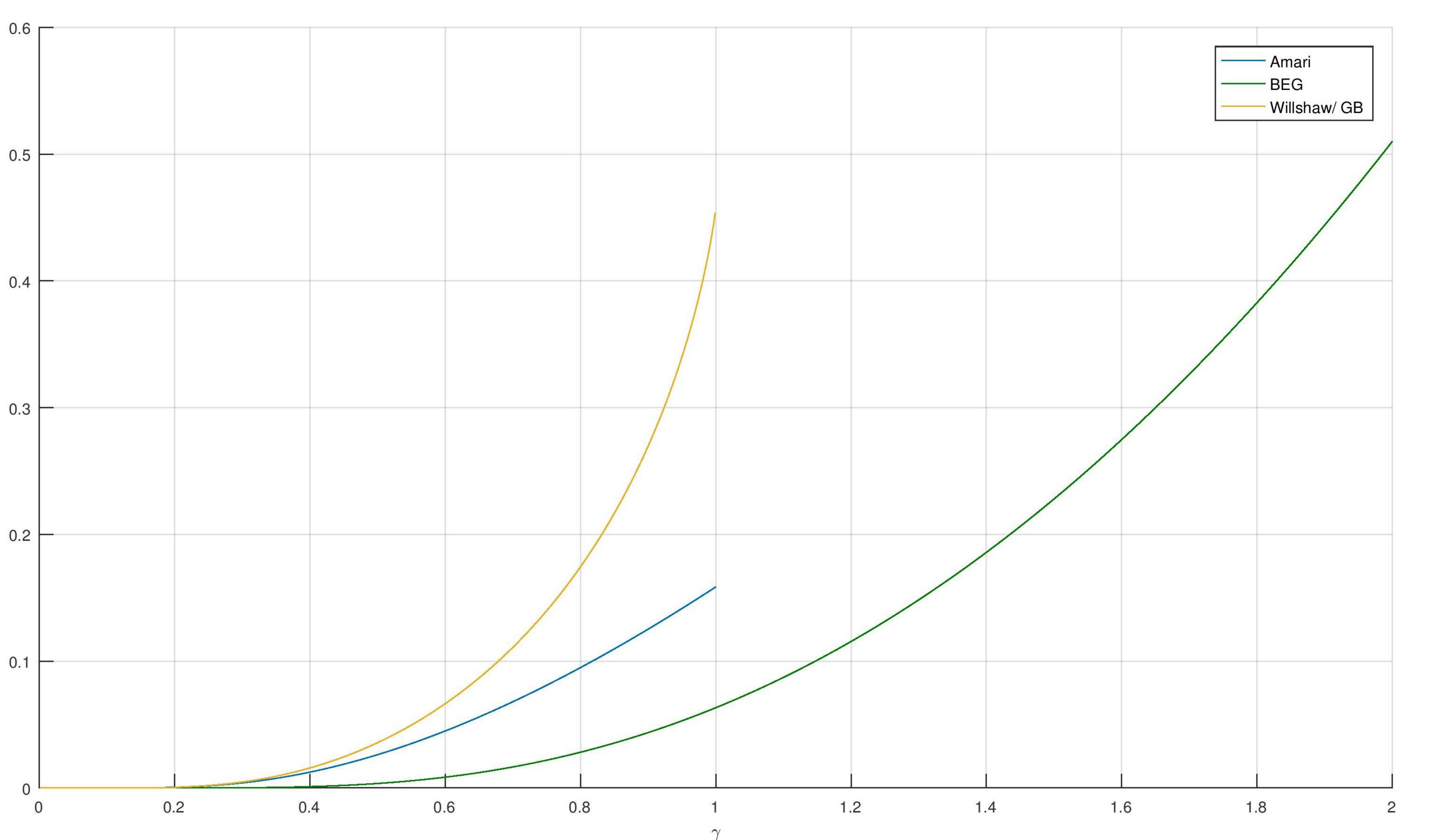}\caption{The storage capacities of sparse associative memories in dependency of $\gamma$ (GB is the model by Gripon and Berrou)}\label{fig1}
\end{figure}

\bibliographystyle{abbrv}

\bibliography{LiteraturDatenbank}
\end{document}